\documentclass[10 pt,twoside]{amsart} 
\usepackage{hyperref,pifont,amssymb,amsmath,diagrams,lipsum}
\usepackage[stable]{footmisc}

\DeclareMathOperator{\re}{Re}

\DeclareMathOperator{\0}{O}
\DeclareMathOperator{\lo}{o}
\DeclareMathOperator{\End}{End}

\begin{document}

\numberwithin{equation}{section}
\newtheorem{thm}{Theorem}[section]  
\newtheorem*{thmA}{Theorem A}
\newtheorem*{thmB}{Theorem B}
\newtheorem{sublem}{Lemma}[thm]
\newtheorem{lem}[thm]{Lemma} 
\newtheorem{cor}[thm]{Corollary}
\newtheorem{prop}[thm]{Proposition}
\theoremstyle{remark}
\newtheorem{rem}[thm]{Remark}
\newtheorem{ex}[thm]{Example}
\theoremstyle{definition}
\newtheorem{Def}[thm]{Definition}

\newcommand{\R}{\mathbb{R} }
\newcommand{\C}{\mathbb{C} }
\newcommand{\Z}{\mathbb{Z}}
\newcommand{\p}{\mathbb{P}}
\newcommand{\ol}{\overline}
\newcommand{\N}{\mathbb{N}}
\newcommand{\ak}{\re(u_k)}
\newcommand{\sk}{\re(v_k)}
\newcommand{\Oxy}{\Omega^{(x,y)}} 
\newcommand{\Oxyh}{\Omega^{(x,y)}_0} 
\newcommand{\Ouv}{\Omega^{(u,v)}} 
\newcommand{\Ou}{\Omega^{u}} 
\newcommand{\Ouvh}{\Omega^{(u,v)}_0} 
\newcommand{\Ouw}{\Omega^{(u,\omega)}} 
\newcommand{\Ouwh}{\Omega^{(u,\omega)}_0} 
\newcommand{\Otwh}{\Omega^{(\tau,\omega)}_0} 
\newcommand{\wt}{\widetilde}
\newcommand{\Szw}{\Sigma^{(z,w)}}
\newcommand{\Suv}{\Sigma^{(u,v)}}
\newcommand{\Sty}{\Sigma^{(t,y)}}
\newcommand{\Sxyo}{\Sigma_0^{(x,y)}}
\newcommand{\Sxy}{\Sigma^{(x,y)}}
\newcommand{\s}{\Sigma}
\newcommand{\Pt}{\widetilde{P}(u,v)}
\newcommand{\Qt}{\widetilde{Q}(u,v)}
\newcommand{\Rt}{\widetilde{R}}
\newcommand{\h}{h}
\newcommand{\Gt}{\widetilde{G}}
\newcommand{\fo}{f_0}
\newcommand{\fa}{f_1}
\newcommand{\fb}{f_2}
\newcommand{\fc}{f_3}
\newcommand{\Id}{\operatorname{Id}}
\newcommand{\Arg}{\operatorname{Arg}}
\newcommand\blfootnote[1]{%
  \begingroup
  \renewcommand\thefootnote{}\footnote{#1}%
  \addtocounter{footnote}{-1}%
  \endgroup
}

\title{Attracting domains of maps tangent to the identity whose only characteristic direction is non-degenerate}
\author{Sara Lapan}
\begin{abstract}We prove that a holomorphic fixed point germ in two complex variables, tangent to the identity, and whose only characteristic direction is non-degenerate, has a domain of attraction on which the map is conjugate to a translation.  In the case of a global automorphism, the corresponding domain of attraction is a Fatou-Bieberbach domain.
\end{abstract}
\date{\today}
\maketitle
\pagestyle{myheadings}
\markright{Attracting domain}

\section*{Introduction}
In this paper, we will be studying holomorphic fixed point germs on $\C^2$ tangent to the identity whose only characteristic direction is non-degenerate.  See \S\ref{prelim} for definitions.

\begin{thmA}
\label{thmA}
Let $f$ be a holomorphic fixed point germ on $\C^2$ tangent to the identity whose only characteristic direction at the fixed point is non-degenerate.  Then there exists a domain $\Omega\subset\C^2$, with the fixed point on the boundary of $\Omega$, that is invariant under $f$ and on which $f$ is conjugate to a translation $(\tau,\omega)\mapsto (\tau,\omega+1)$.  
\end{thmA}
This is similar to a theorem proved by Hakim \cite{H1} (see also \cite{W2,AR}), however the theorem does not apply here because the director of $f$ is zero (see \S 1).  Vivas \cite{V2} also proved a similar result in the case of a degenerate characteristic direction.  These results are generalizations of the one-dimensional Leau-Fatou flower theorem which applies to a holomorphic fixed point germ $f$ on $\C$ tangent to the identity \cite{CG,M}.  \\ 

Theorem A partially answers questions raised by Abate about the quadratic map ($1_{11}$) in \cite{A1}.  This will be discussed more in the next section.

\begin{thmB}
\label{thmB}
Suppose $f$ as in Theorem A is an automorphism of $\C^2$.  Then there exists a Fatou-Bieberbach domain $\Sigma\subset\C^2$ with the fixed point in its boundary, that is invariant under $f$, and on which $f$ is conjugate to translation. 
\end{thmB}
We know that such automorphisms exist by results proven in \cite{BF} and \cite{W1}.  Recall that a Fatou-Bieberbach domain is a proper subdomain of $\C^n$ that is biholomorphically equivalent to $\C^n$. \\

The paper is organized in the following way.  In \S 1, we introduce the main definitions that will be used in this paper and show how to conjugate $f$ from Theorem A to a suitable normal form using a linear change of coordinates.  In \S 2, we perform a coordinate change moving the fixed point from the origin to infinity and we find an invariant region.  In \S 3, we perform another more complicated coordinate change so that $f$ acts on the second coordinate by translation.  The technique we employ to find the coordinate change is similar to that used in the degenerate case studied in \cite{V1}, but is more involved because we solve a system of (instead of just one) differential equations.  In \S 4, we perform a final coordinate change so that $f$ acts as the identity on the first coordinate and translation on the second.  In each section we find invariant domains so that in \S 4 we have finished showing Theorem A.  In \S 5, we assume that $f$ is an automorphism and extend our domain from \S 4 to one that is biholomorphic to $\C^2$, concluding with Theorem B. 

\subsection*{Acknowledgements}
The author would like to thank Mattias Jonsson for his guidance in choosing and studying this problem.  The author was supported in part by the NSF.

\section{Preliminaries}
\label{prelim}
Denote by $\End(\C^n,O)$ the set of holomorphic germs of self-maps of $\C^n$ that fix the origin.  Every $f\in\End(\C^n,O)$ can be written in the form:
$$f(x_1,\ldots,x_n)=P_1(x_1,\ldots,x_n)+P_2(x_1,\ldots,x_n)+\cdots,$$
where $P_j=(P_j^1,\ldots,P_j^n)$ and each $P_j^l$ is a homogeneous polynomial of degree $j$.  
\begin{Def}A germ $f\in\End(\C^n,O)$ is \textit{tangent to the identity} if $P_1=\Id$.  If $P_k\not\equiv 0$ and $P_j\equiv 0$  for $1<j<k$, then $k$ is the \textit{order} of $f$. \end{Def}
\begin{Def}Let $f\in\End(\C^n,O)$ be tangent to the identity of order $k$.  A \textit{characteristic direction} for $f$ is the projection in $\p^{n-1}$ of any $v\in\C^n\setminus\{O\}$ such that $P_k(v)=\lambda v$ for some $\lambda\in\C$; the characteristic direction is \textit{degenerate} if $\lambda=0$ and \textit{non-degenerate} if $\lambda\neq 0$.\end{Def}
\begin{Def}Given $f\in\End(\C^n,O)$ tangent to the identity with a non-degenerate characteristic direction $[v]\in\p^{n-1}$, the eigenvalues of the linear operator $D(P_k)_{[v]}-\Id:T_{[v]}\p^{n-1}\to T_{[v]}\p^{n-1}$ are the \textit{directors} of $[v]$.\end{Def}

Let $f\in\End(\C^n,O)$ be tangent to the identity with a non-degenerate characteristic direction all of whose directors have strictly positive real part.  Hakim proved in \cite[Theorem~1.8]{H1} that there is a subdomain of $\C^n$ in which every point is attracted to the origin along the given characteristic direction and $f$ is conjugate to a translation on that domain.  Conversely, if $g\in\End(\C^n,O)$ is tangent to the identity and has an attracting domain around a non-degenerate characteristic direction where all orbits converge to the origin along that direction, then all of the directors have non-negative real part \cite[Corollary~8.11]{AR}.  If, in addition to Hakim's previous constraints, $f$ is a biholomorphism of $\C^n$, then there is a Fatou-Bieberbach domain of $\C^n$ (i.e. the domain is biholomorphic to $\C^n$) in which every point is attracted to the origin along the given characteristic direction and $f$ is conjugate to a translation on that domain \cite[Theorem~1.10]{H1}.  Before Hakim proved this result, Weickert in \cite[Theorem 1]{W2} showed the existence of biholomorphisms tangent to the identity with such an invariant domain in dimension 2.  \\

Vivas showed similar results \cite[Theorem~1\mbox{ and }2]{V2} for maps with specific types of \emph{degenerate} characteristic direction in dimension 2.   In particular, let $f\in\End(\C^2,O)$ be tangent to the identity with a degenerate characteristic direction that satisfies specific properties (for instance it is irregular\footnote{After completing this paper, the author became aware of an updated version of Vivas's paper~\cite{V2}, in which the degenerate assumption from~\cite[Theorem 1]{ V2} was removed.  In Theorem A and B, we assume that $f$ has only one characteristic direction and that this direction is non-degenerate.  From this assumption, we are able to write $f$ in the form~\eqref{fo}, where it becomes apparent that the direction must be irregular (see~\cite{V2} for definition), at this point~\cite[Theorem 1]{V2} may be used to complete the proofs of Theorem A and B.} 
 or it is Fuchsian and the map satisfies other constraints).  Then there is a domain that is attracted to the origin along that characteristic direction.  If, in addition, $f$ is a biholomorphism of $\C^2$, then $f$ has a domain that is attracted to the origin along that direction which is biholomorphic to $\C^2$.  On the other hand,  Stens\o nes and Vivas \cite{SV} showed a negative result: for $n\geq 3$ there exists a biholomorphism of $\C^n$ tangent to the identity whose basin of attraction is biholomorphic to $(\C\setminus\{0\})^{n-2}\times\C^2$. \\

In order to use Theorem A to help show Theorem B, we will use the following theorem due independently to Weickert \cite[Theorem~2.1.1]{W1} and Buzzard-Forstneric \cite[Theorem~1.1]{BF}:
\begin{thm}\label{WThm2.1.1}
Let $P=(P_1,\ldots,P_n), n\geq 2$, be a holomorphic polynomial self-map of $\C^n$ with $P'(0)$ invertible.  Let $d\geq\max_i\{\deg(P_i)\}$.  Then there exists $\psi:\C^n\to\C^n$, a biholomorphism, such that $|\psi(z)-P(z)|=\lo\left(|z|^d\right)$ near the origin.\end{thm}

For the rest of this paper we will restrict to dimension $2$.  Throughout the paper we will use the notation $\pi_j$ to denote projection onto the $j$th coordinate.

\begin{lem}
\label{linear}
Let $f\in\End(\C^2,O)$ be tangent to the identity of order $k$ with only one characteristic direction at the origin and this direction is non-degenerate.  Then $f$ is linearly conjugate to the map
\begin{equation}\label{fo}
\fo(x,y)=\left(x\left(1+xyR(x,y)+y^{k-1}\right)+P(x,y),y\left(1+xyR(x,y)+y^{k-1}\right)+x^k+Q(x,y)\right),
\end{equation}
where $P(x,y), Q(x,y)$ are convergent power series vanishing to order at least $k+1$ at the origin and $R(x,y)$ is a homogeneous polynomial of degree $k-3$ with the constraint $R\equiv 0$ if $k\leq 2$. \end{lem}

\begin{proof}
We can write $f(x,y)$ as a sum of its homogenous polynomials,
$$f(x,y)=(x,y)+\sum_{j=k}^\infty P_j(x,y),$$
where $P_j(x,y)$ are homogeneous polynomials of degree $j$ and $P_k\not\equiv 0$.  We assume that $[0:1]$ is the characteristic direction of $f$ since via a linear conjugation of $f$ we can move the characteristic direction of $f$ to $[0:1]$ without changing the degree of any of the $P_j(x,y)$.  We can write the $k$-th degree polynomial as:
$$P_k(x,y)=\left(\sum_{j=0}^{k} a_j x^{k-j}y^j,\sum_{j=0}^{k} b_j x^{k-j}y^j\right),$$
where $a_j,b_j\in\C$.  Since $[0:1]$ is the only characteristic direction of $f$, $P_k[x:y]\neq[x:y]$ for all $x\neq 0$.  This restricts the possible values of $\{a_j,b_j\}$: 
$$P_k(0,1)=(a_k,b_k)\mbox{ and }P_k(1,0)=(a_0,b_0), \mbox{ therefore }a_k=0, b_k\neq 0\mbox{ and } b_0\neq 0,$$
$$P_k(x,1)=\left(\sum_{j=0}^{k-1} a_j x^{k-j},\sum_{j=0}^{k} b_j x^{k-j}\right)
	=\left(x\sum_{j=0}^{k-1} a_j x^{k-1-j},\sum_{j=0}^{k} b_j x^{k-j}\right) 
	\neq \lambda(x,1)$$
for any $\lambda\in\C$.  Since
$$P_k[x:1]=[x:1]\Leftrightarrow \sum_{j=0}^{k-1} a_j x^{k-1-j}=\sum_{j=0}^{k} b_j x^{k-j} \Leftrightarrow -b_0x^k+\sum_{j=1}^{k} (a_{j-1}-b_j) x^{k-j}=0$$
and the first condition is true only when $x=0$, the last condition must only be true when $x=0$.  Thus $a_{j-1}=b_j, \forall 1\leq j\leq k$.  
We can now re-write $P_k(x,y)$ as:
$$P_k(x,y)=\left(x\sum_{j=0}^{k-1} a_j x^{k-1-j}y^j,y\sum_{j=0}^{k-1} a_{j} x^{k-1-j}y^j+b_0x^k\right):=(xS(x,y),yS(x,y)+b_0x^k).$$
Now that we have a more explicit form for $P_k$, we want to simplify it further using linear conjugation.  Let $l$ be a linear map that fixes $[0:1]$.  Then we can write $l$ as:
\begin{equation}\label{l}
l(x,y):=(ax,cx+dy)\mbox{ and }l^{-1}(x,y)=\frac{1}{ad}(dx,-cx+ay),\mbox{ where }ad\neq 0.
\end{equation}
\begin{align*}
l^{-1}\circ P_k\circ l(x,y)
	&=l^{-1}(axS(ax,cx+dy),(cx+dy)S(ax,cx+dy)+b_0a^k x^k) \\
	&=\left(xS(ax,cx+dy),yS(ax,cx+dy)+\frac{b_0 a^k}{d}x^k\right)\\
S(ax,cx+dy)
	&=\sum_{j=0}^{k-1} a_j a^{k-1-j} x^{k-1-j}(cx+dy)^j \\
	&=x^{k-1}\left(\sum_{j=0}^{k-1}a_j a^{k-1-j}c^j\right)+xy(\cdots)+y^{k-1}(a_{k-1}d^{k-1})
\end{align*}
We can choose $a,c,d$ so that (1) $\frac{b_0 a^k}{d}=1$, (2) $a_{k-1}d^{k-1}=1$, and (3) $\sum_{j=0}^{k-1}a_j a^{k-1-j}c^j=0$.  Therefore $P_k(x,y)$ is linearly conjugate to:
$$\left(x\left(xyR(x,y)+y^{k-1}\right),y\left(xyR(x,y)+y^{k-1}\right)+x^k\right),$$
where $R\equiv 0$ if $k\leq 2$, otherwise $R(x,y)=\frac{1}{xy}\left(S(ax,cx+dy)-y^{k-1}\right)$ is a homogeneous polynomial of degree $k-3$.  Let
$$(P(x,y),Q(x,y)):=l^{-1}\circ\sum_{j=k+1}^\infty P_j\circ l(x,y).$$
Since $(P_j)$ are convergent power series in a neighborhood of the origin, so are $P,Q$.   \end{proof}

Abate in \cite{A1} studied quadratic maps tangent to the identity up to holomorphic conjugacy.  He showed that for quadratic self maps of $\C^2$ tangent to the identity, holomorphic conjugacy was equivalent to linear conjugacy and used this along with the number of characteristic directions of the maps to classify all such maps.  In addition, Ueda \cite{U,W2} and Rivi in \cite{R} also classified such maps.  In this paper, if we assume that the map $f$ is quadratic with no terms of higher degree, then $f_0$ is the same map as Abate called ($1_{11}$) in \cite{A1}, namely
$$f_0(x,y)=\left(x(1+y),y(1+y)+x^2\right)=(x+xy,y+y^2+x^2).$$
 
The next lemma shows that Hakim's result does not apply to $f$.

\begin{lem}The real part of the director of $f$ at $[0:1]$ is zero.\end{lem}
\begin{proof}
Let $U:=\left\{[x_0:x_1]\in\C\p^1 \ \big| \ x_1\neq 0 \right\}$ and define $\pi:U\to\C$ by $\pi([x_0:x_1])=\frac{x_0}{x_1}$. Define 
$$g(x):=\pi\circ \widehat{P_k} \circ\pi^{-1}(x)=\pi\left[xS(x,1):S(x,1)+b_0x^k\right]=\frac{xS(x,1)}{S(x,1)+b_0x^k},$$
where $x\in\C$ and $ \widehat{P_k}:\C\p^1\to\C\p^1$ by $[v]\mapsto[P_k(v)]$.  So for $x\in\C$, $D(\widehat{P_k})_{[x:1]}-\Id=g'(x)-\Id$.  Therefore the director of $f$ at $[0:1]$ is the value of $g'(0)-\Id$.  Since
$$g'(x)=\frac{S(x,1)^2+b_0x^kS(x,1)+b_0x^{k+1}S'(x,1)-kb_0x^kS(x,1)}{\left(S(x,1)+b_0x^k\right)^2},$$
$g'(0)-\Id=0$.  Therefore the director of $f$ at $[0:1]$ is zero.
\end{proof}

\section{Invariant Region}
We want to find a domain of attraction for the map $\fo$, which is equivalent to finding one for $f$.  For $xy\neq 0$, let
$$(u,v):=\psi_0(x,y):=\left(a\frac{y^k}{x^k},\frac{b}{y^{k-1}}\right),$$
where $a=-\frac{k-1}{k},b=-\frac{1}{k-1}$.  Define
\begin{equation}\label{Ouveq}
\Ouv_{R,\delta,\theta}:=\left\{ (u,v)\in\C^2 \ \bigg| \  \re(u)>R, |u|^{\frac{(k-1)(k+1)}{k}}<\delta|v|, |\Arg (u)|<\theta,|\Arg (v)|<\frac{k-1}{k}\theta\right\} 
\end{equation}
and
\begin{equation}\label{Ou}
\Ou_{R,\theta}:=\left\{u\in\C \ \bigg| \ \re(u)>R, |\Arg (u)|<\theta\right\}=\pi_1\left(\Ouv_{R,\delta,\theta}\right), 
\end{equation}
for any 
$$0<\theta<\frac{\pi}{4},\qquad 0<\delta\ll1,\qquad\mbox{and}\qquad R\gg0.$$
Fix $R,\delta,\theta$ satisfying the above conditions to define
$$\Ouv:=\Ouv_{R,\delta,\theta}\qquad\mbox{ and }\qquad \Ou:=\Ou_{R,\theta}.$$
Now we can define an inverse to $\psi_0$ restricted to the domain $\Ouv$:
$$(x,y):=\psi_0^{-1}(u,v)=\left(\left(\frac{a}{u}\right)^\frac{1}{k}\left(\frac{b}{v}\right)^\frac{1}{k-1},\left(\frac{b}{v}\right)^\frac{1}{k-1}\right),$$
where we choose the $\frac{1}{k},\frac{1}{k-1}$ roots that map 1 to 1.  Therefore $\psi_0:\psi_0^{-1}\left(\Ouv\right)\to\Ouv$ is a biholomorphism.  Note that $0\in\partial\left(\psi_0^{-1}\left(\Ouv\right)\right)$.

\begin{prop}\label{Shrink}
Given $R'<R$ and $\theta<\theta'<\frac{\pi}{4}$, $\exists\kappa>0$ such that if $u\in\Ou_{R,\theta}$, then $B\left(u,\kappa|u|\right)\subset\Ou_{R',\theta'}$.
Furthermore, given $\alpha\neq0$  and a holomorphic function $F$ on $\Ou_{R',\theta'}$ satisfying the bound $F=\0(u^\alpha)$, then $\forall n\in\N$, the $n$th derivative satisfies the bound $F^{(n)}=\0\left(u^{\alpha-n}\right)$ on $\Ou_{R,\theta}$.
\end{prop}
\begin{proof}
Given $u\in\Ou_{R,\theta}$,
$$\re(u)-R'>\left(1-\frac{R'}{R}\right)\re(u)>\frac{1}{2}\left(1-\frac{R'}{R}\right)|u|$$
and
$$\sin\left(\theta'-\Arg(u)\right)|u|>\sin(\theta'-\theta)|u|.$$
Hence for any $0<\kappa\leq\min\left\{\frac{1}{2}\left(1-\frac{R'}{R}\right),\sin\left(\theta'-\theta\right)\right\}$, the disk $B\left(u,\kappa|u|\right)\subset \Ou_{R',\theta'}$.  On $\Ou_{R',\theta'}$, $F(\zeta)=\0(\zeta^\alpha)$ and $\exists C>0$ such that 
$|F(\zeta)|<C|\zeta|^\alpha\leq
\begin{cases}
	C(1+\kappa)^\alpha|u|^\alpha, & \mbox{if }\alpha\geq 0 \\
	C(1-\kappa)^\alpha|u|^\alpha, & \mbox{if }\alpha< 0 
\end{cases}$.  Therefore, $\forall n\in\N$,
$$\bigg|F^{(n)}(u)\bigg| \leq \frac{n!}{2\pi}\Bigg| \int_{|\zeta-u|=\kappa|u|}\frac{F(\zeta)}{(\zeta-u)^{n+1}}d\zeta\Bigg|
\leq\frac{n!}{(\kappa|u|)^n}\sup_{|\zeta-u|=\kappa|u|}|F(\zeta)|
=\0\left(u^{\alpha-n}\right),$$
where we used Cauchy estimates to get the first inequality.
\end{proof}

\begin{rem}\label{Ouv}
On several occasions, we will adjust $R,\delta,\theta$ to shrink the domain $\Ouv$ (or $\Ou$).  In particular, we will choose $R',\delta',\theta'$ that depend on $R,\delta,\theta$ so that by making $R$ large enough and $\delta,\theta $ small enough the domain $\Ouv_{R',\delta',\theta'}$ (or $\Ou_{R',\theta'}$) satisfies all of the properties that had been shown for $\Ouv$ (respectively, $\Ou$) and $\Ouv_{R',\delta',\theta'}\supsetneq\Ouv$ (respectively, $\Ou_{R',\theta'}\supsetneq\Ou$).  We will use this and Proposition~\ref{Shrink} to find domains on which a holomorphic function is defined as well as similar subdomains on which we can bound the derivatives of that holomorphic function.\end{rem}

Let $\fa:=\psi_0\circ \fo\circ\psi_0^{-1}$.  For any $(u,v)\in\Ouv$, denote the $n$-th iterate of the map after the coordinate change by $\fa^n(u,v):=(u_n,v_n)$, where $\fa^0(u,v)=(u,v)$.  Later in this section we will prove the following results on invariance of $\Ouv$ and size of $(u_n,v_n)$.

\begin{lem}\label{Ouvinvariant}
$\Ouv$ is invariant under $\fa$.\end{lem}

\begin{lem}\label{uvBound}
For any $(u,v)\in\Ouv$ and any positive integer $n$,
\begin{equation}\label{vineq}
\re(v)+\frac{3n}{2} \geq\re(v_n)\geq\re(v)+\frac{n}{2} 
\end{equation}
and 
\begin{equation}\label{uineq}
\re(u)+3\log\left(1+\frac{n}{\re(v)}\right) \geq \re(u_n) \geq \re(u)+\frac{1}{6}\log\left(1+\frac{n}{\re(v)}\right).
\end{equation}
\end{lem}

It follows from Lemma~\ref{uvBound} that for any $(u,v)\in\Ouv$,
\begin{align}\label{|uv|Bound}
2|v|+3n &
	\geq|v_n|\geq \frac{|v|+n}{2}, \mbox{ and } \\
2|u|+6\log\left(1+\frac{n}{\re(v)}\right)&
	\geq |u_n|\geq \frac{|u|}{2}+\frac{1}{6}\log\left(1+\frac{n}{\re(v)}\right).\notag
\end{align}

In order to simplify the coordinate change, we make the following definitions:
\begin{align}
\Rt(u) & :=\left(\frac{v}{b}\right)^\frac{k-3}{k-1}R(x,y)=R\left(\left(\frac{a}{u}\right)^\frac{1}{k},1\right), \\
\Pt &:= \left(\frac{v}{b}\right)^\frac{k+1}{k-1}P(x,y)=\left(\frac{v}{b}\right)^\frac{k+1}{k-1} P\left(\left(\frac{a}{u}\right)^\frac{1}{k}\left(\frac{b}{v}\right)^\frac{1}{k-1},\left(\frac{b}{v}\right)^\frac{1}{k-1}\right), \notag \\
\Qt &:= \left(\frac{v}{b}\right)^\frac{k+1}{k-1}Q(x,y)=\left(\frac{v}{b}\right)^\frac{k+1}{k-1} Q\left(\left(\frac{a}{u}\right)^\frac{1}{k}\left(\frac{b}{v}\right)^\frac{1}{k-1},\left(\frac{b}{v}\right)^\frac{1}{k-1}\right), \notag \\
\h(u,v) &:= kb^\frac{k}{k-1}\left(\frac{\Qt}{u^\frac{1}{k}}-\frac{\Pt}{a^\frac{1}{k}}\right):=\sum_{j=0}^\infty\frac{\h_j(u)}{v^\frac{j}{k-1}}.\notag
\end{align}
Then $\widetilde{P},\widetilde{Q},\h$ are convergent power series in $u^{-\frac{1}{k}},v^{-\frac{1}{k-1}}$, $\h_j$ is a convergent power series in $u^{-\frac{1}{k}}$, and $\h_j,\Rt$ are holomorphic on $\Ou$.  
Now we find an expression for $u_1$:
\begin{align*}
u_1 
&=u\left(1+\frac{\frac{x^k}{y}+\frac{Q(x,y)}{y}-\frac{P(x,y)}{x}}{1+xyR(x,y)+y^{k-1}+\frac{P(x,y)}{x}}\right)^k \\
&=u\left(1+\frac{1}{v}\frac{\frac{ab}{u}+b\left(\frac{b}{v}\right)^\frac{1}{k-1}\Qt-b\left(\frac{b}{v}\right)^\frac{1}{k-1}\left(\frac{u}{a}\right)^\frac{1}{k}\Pt}{1+\frac{b}{v}\left(\left(\frac{a}{u}\right)^\frac{1}{k}\Rt(u)+1+\left(\frac{b}{v}\right)^\frac{1}{k-1}\left(\frac{u}{a}\right)^\frac{1}{k}\Pt\right)}\right)^k \\
&=u\left(1+\frac{1}{v}\left(\frac{ab}{u}+\frac{u^\frac{1}{k}}{v^\frac{1}{k-1}}\frac{\h(u,v)}{k}\right)\left[
1+\0\left(\frac{1}{v}\right)\right]\right)^k \\
&= u\left(1+\frac{k}{v}\left(\frac{ab}{u}+\frac{u^\frac{1}{k}}{v^\frac{1}{k-1}}\frac{\h(u,v)}{k}\right)
+\0\left(\frac{1}{uv^2},\frac{u^\frac{1}{k}}{v^{2+\frac{1}{k-1}}}\right)\right) \\
&=u+\frac{1}{v}+\frac{u^\frac{k+1}{k}}{v^\frac{k}{k-1}}\h(u,v)+\0\left(\frac{1}{v^2}\right) \\
&=u+\frac{1}{v}+\frac{u^\frac{k+1}{k}}{v^\frac{k}{k-1}}\sum_{j=0}^{k-2}\frac{\h_j(u)}{v^\frac{j}{k-1}}+\0\left(\frac{1}{v^2}\right),
 \end{align*}
 
Similarly:

\begin{align*}
v_1 
&=v\left(1+xyR(x,y)+y^{k-1}+\frac{x^k}{y}+\frac{Q(x,y)}{y} \right)^{-(k-1)} \\
&=v\left(1+\frac{b}{v}\left[\left(\frac{a}{u}\right)^\frac{1}{k}\Rt(u)+1+\frac{a}{u}+\left(\frac{b}{v}\right)^\frac{1}{k-1}\Qt\right]\right)^{-(k-1)} \\
&=v\Bigg(1-(k-1)\frac{b}{v}\left[1+\left(\frac{a}{u}\right)^\frac{1}{k}\Rt(u)+\frac{a}{u}+\left(\frac{b}{v}\right)^\frac{1}{k-1}\Qt\right] \\
&\qquad\quad +k(k-1)\frac{b^2}{v^2}\left[1+\left(\frac{a}{u}\right)^\frac{1}{k}\Rt(u)+\frac{a}{u}+\left(\frac{b}{v}\right)^\frac{1}{k-1}\Qt\right]^2+\0\left(\frac{1}{v^3}\right)\Bigg) \\
&=v+1+\left(\frac{a}{u}\right)^\frac{1}{k}\Rt(u)+\frac{a}{u}+\left(\frac{b}{v}\right)^\frac{1}{k-1}\Qt+\frac{k}{k-1}\frac{1}{v}\left[1+\left(\frac{a}{u}\right)^\frac{1}{k}\Rt(u)+\frac{a}{u}\right]^2+\0\left(\frac{1}{v^\frac{k}{k-1}}\right) \\
&=v+1+\sum_{j=0}^{k-1}\frac{g_j(u)}{v^\frac{j}{k-1}}+\0\left(\frac{1}{v^\frac{k}{k-1}}\right),
\end{align*}
where $g_0$ is a polynomial in $u^{-\frac{1}{k}}$ with no constant term and $g_j$ are power series in $u^{-\frac{1}{k}}$.  We will frequently use these properties of the $\{g_j\}$ in what follows.  \\

To summarize, for $(u,v)\in\Ouv$ we have derived the following equations for $u_1,v_1$:
\begin{align}
u_1 &
	=u+\frac{1}{v}+\frac{u^\frac{k+1}{k}}{v^\frac{k}{k-1}}\h(u,v)+\0\left(\frac{1}{v^2}\right)
	=u+\frac{1}{v}+\frac{u^\frac{k+1}{k}}{v^\frac{k}{k-1}}\sum_{j=0}^{k-2}\frac{\h_j(u)}{v^\frac{j}{k-1}}+\0\left(\frac{1}{v^2}\right) \label{u} \\
v_1 &= v+1+\sum_{j=0}^{k-1}\frac{g_j(u)}{v^\frac{j}{k-1}}+\0\left(\frac{1}{v^\frac{k}{k-1}}\right)
	=v+1+\0\left(\frac{1}{u^\frac{1}{k}},\frac{1}{v^\frac{1}{k-1}}\right).\label{v}
\end{align}

\begin{proof}[Proof of Lemma~\ref{Ouvinvariant}]
Fix any $(u,v)\in\Ouv$.  First we show that $\re(u_1)>R$.
\begin{align*}
\re(u_1) &= \re(u)+\re\left(\frac{1}{v}\left[1+\frac{u^\frac{k+1}{k}}{v^\frac{1}{k-1}}\h(u,v)+\0\left(\frac{1}{v}\right)\right]\right) \\
	&>\re(u)+\frac{1}{2|v|}-\frac{2\delta^\frac{1}{k-1} C}{|v|}
	 >\re(u)>R
\end{align*}
where $C$ is some constant such that $|\h(u,v)|<C$, $R$ is sufficiently large and $\delta$ is chosen to be suitably small.  Next we check that the arguments of $u_1,v_1$ remain small enough.  
\begin{align*}
|\Arg (v_1)| &\leq\max\left\{|\Arg(v)|,\left|\Arg\left(1+\sum_{j=0}^{k-1}\frac{g_j(u)}{v^\frac{j}{k-1}}+\0\left(\frac{1}{v^\frac{k}{k-1}}\right)\right)\right|\right\}\leq\frac{k-1}{k}\theta,
\end{align*}
for $R$ large enough.
\begin{align*}
|\Arg(u_1)| &\leq\max\left\{|\Arg(u)|,
|\Arg(v)|+\left|\Arg\left(1+\frac{u^\frac{k+1}{k}}{v^\frac{1}{k-1}}\h(u,v)+\0\left(\frac{1}{v}\right)\right)\right|\right\}<\theta,
\end{align*}
for $\delta$ suitably small.  Finally we verify that $|u_1|$ remains small enough relative to $|v_1|$. 
\begin{align*}
|u_1|^\frac{k^2-1}{k} &<|u|^\frac{k^2-1}{k}\left|1+\frac{1}{uv}\left(1+\frac{u^\frac{k+1}{k}}{v^\frac{1}{k-1}}\h(u,v)+\0\left(\frac{1}{v}\right)\right)\right|^\frac{k^2-1}{k} \\
	&<\delta|v|\left|1+\frac{k^2-1}{k}\frac{1}{uv}\left(1+\frac{u^\frac{k+1}{k}}{v^\frac{1}{k-1}}\h(u,v)+\0\left(\frac{1}{v}\right)\right)\right| \\
	&<\delta|v|\bigg|1+\frac{1}{2v}\bigg|<\delta|v_1|,
\end{align*}
for large enough $R$.  Therefore, $\fa\left(\Ouv\right)\subset\Ouv$.
\end{proof}

\begin{proof}[Proof of Lemma~\ref{uvBound}]
\begin{align*}
\re(v_n) &=\re(v)+n+\sum_{j=0}^{n-1}\0\left(u_j^{-\frac{1}{k}},v_j^{-\frac{1}{k-1}}\right)\\
u_n &= u+\sum_{j=0}^{n-1}\frac{1}{v_j}\left[1+\frac{u_j^\frac{k+1}{k}}{v_j^\frac{1}{k-1}}\h(u_j,v_j)+\0\left(\frac{1}{v_j}\right)\right] 
\end{align*}
In the expression for $\re(v_n)$, the $\0$-terms are bounded independently of $j$ and so we can choose $R$ sufficiently large that, for each $j$, the term $\0\left(u_j^\frac{-1}{k},v_j^\frac{-1}{k-1}\right)$ is bounded above by $\frac{1}{2}$.  Hence, $|\re(v_n-v)-n|\leq\frac{n}{2}$.  Now we find bounds on $\frac{|u_n|}{|v_n|^\alpha}$ and $\re(u_n)$:
\begin{align*}
\re(u_n-u) 
	&\leq \sum_{j=0}^{n-1}\frac{1}{|v_j|}
		\left(1+2\delta^\frac{1}{k-1}C\right) 
	\leq \sum_{j=0}^{n-1}\frac{\frac{3}{2}}{\re(v)+\frac{1}{2}j}
	\leq \int_{-1}^{n-1} \frac{3dx}{x+2\re(v)}<3\log\left(1+\frac{n}{\re(v)}\right) \\
\re(u_n-u)
	&\geq \sum_{j=0}^{n-1}\frac{1}{|v_j|}
		\left(1-2\delta^\frac{1}{k-1}C\right)
	\geq \sum_{j=0}^{n-1}\frac{\frac{1}{4}}{\re(v)+\frac{3}{2}j} 
	\geq \frac{1}{6} \int_0^n\frac{dx}{x+\frac{2}{3}\re(v)} 
	\geq\frac{1}{6}\log\left(1+\frac{n}{\re(v)}\right)
\end{align*}
Therefore we have the desired bounds on $\re(u_n),\re(v_n)$.
\end{proof}

\section{Fatou Coordinate}
In this section we will perform a coordinate change to simplify the expression for~\eqref{v}.   For any $(u,v)\in\Ouv$ we define:
$$(u,w):=\widetilde{\psi}_1(u,v):=\left(u,v\left[1+\sum_{j=0}^{k-1}\frac{\phi_j(u)}{v^\frac{j}{k-1}}\right]\right),$$
where $\{\phi_j\}$ are holomorphic functions in $\Ou$ that we will define later as solutions of certain differential equations.  
Let $w_n:=\pi_2\circ\widetilde{\psi}_1\circ f_1^n(u,v)=\pi_2\circ\widetilde{\psi}_1(u_n,v_n)$.  The goal of this section is to show that the sequence $\left(w_n-n\right)_{n=1}^\infty$ converges uniformly to a Fatou coordinate $\omega$, i.e. $\omega\circ f=\omega+1$.  Also, the map $(u,v)\mapsto(u,\omega(u,v))$ defines a coordinate change, i.e. a biholomorphism onto its image.  
Before introducing the Fatou coordinate, we need to simplify the expression for $w_1$ in terms of $u,w$ and define the holomorphic functions $\left\{\phi_j\right\}$.  Using Taylor series expansion: 
\begin{equation}\label{phij}
\phi_j(u_1)=\phi_j(u)+(u_1-u)\phi_j'(u)+\frac{(u_1-u)^2}{2}\phi_j''\left(\hat{u}\right),
\end{equation}
where $\hat{u}$ is some point on the line between $u$ and $u_1$ and it may vary depending on $j$.  Then
\begin{align*}
w_1 &=v_1+\sum_{j=0}^{k-1}\phi_j(u_1)v_1^{1-\frac{j}{k-1}} \\
	&=v_1+\sum_{j=0}^{k-1}
		\left[\phi_j(u)+\frac{\phi_j'(u)}{v}\left(1+\frac{u^\frac{k+1}{k}}{v^\frac{1}{k-1}}\sum_{l=0}^{k-2}\frac{\h_l(u)}{v^\frac{l}{k-1}}+\0\left(\frac{1}{v}\right)\right)+\0\left(\frac{\phi_j''\left(\hat{u}\right)}{v^2}\right)\right]\cdot \\
	&\qquad\qquad\qquad v^{1-\frac{j}{k-1}}\left[1+\frac{1}{v}\left(1+\sum_{m=0}^{k-1}\frac{g_m(u)}{v^\frac{m}{k-1}}+\0\left(\frac{1}{v^\frac{k}{k-1}}\right)\right)\right]^{1-\frac{j}{k-1}} \\
	&=v_1+\sum_{j=0}^{k-1}v^{1-\frac{j}{k-1}}
		\left[\phi_j(u)+\frac{\phi_j'(u)}{v}\left(1+\frac{u^\frac{k+1}{k}}{v^\frac{1}{k-1}}\sum_{l=0}^{k-2}\frac{\h_l(u)}{v^\frac{l}{k-1}}\right)+\0\left(\frac{\phi_j'(u)}{v^2},\frac{\phi_j''\left(\hat{u}\right)}{v^2}\right)\right]\cdot \\
	&\qquad\qquad\qquad\left[1+\left(1-\frac{j}{k-1}\right)\frac{1}{v}\left(1+\sum_{m=0}^{k-1}\frac{g_m(u)}{v^\frac{m}{k-1}}\right)+
	\frac{(k-1-j)}{v^2}\0\left(j,\frac{1}{v^\frac{1}{k-1}}\right)\right] \\
	&=w+1+\sum_{j=0}^{k-1}\frac{1}{v^\frac{j}{k-1}}\left(g_j(u)+
		\frac{k-1-j}{k-1}(1+g_0(u))\phi_j(u)+\phi_j'(u)\right) \\
		&\qquad+\sum_{j=0}^{k-1}\sum_{m=1}^{k-1}
		\frac{1}{v^\frac{j+m}{k-1}}\left(
		\frac{k-1-j}{k-1}\phi_j(u)g_m(u)+\phi_j'(u)u^\frac{k+1}{k}\h_{m-1}(u)\right) \\
	&\qquad+\sum_{j=0}^{k-1}\0\left(\frac{1}{v^\frac{k}{k-1}},\frac{\phi_0(u)}{v^\frac{k}{k-1}},j(k-1-j)\frac{\phi_j(u)}{v^\frac{k+j-1}{k-1}},\frac{\phi_j'(u)}{v^\frac{k+j-1}{k-1}},\frac{\phi_j''\left(\hat{u}\right)}{v^\frac{k+j-1}{k-1}}\right) \\
\end{align*}
In order to simplify the equation for $w_1$, define $F_j,G_j$ as:
\begin{align}\label{FjGj}
F_j(u)&:=\frac{k-1-j}{k-1}\left(1+g_0(u)\right), \\
G_j(u)&:=-\left[g_j(u)+\sum_{l=0}^{j-1}\left(\frac{k-1-l}{k-1}\phi_l(u)g_{j-l}(u)+\phi_l'(u)u^\frac{k+1}{k}\h_{j-l-1}(u)\right)\right].\notag
\end{align}
for all integers $0\leq j \leq k-1$.  Therefore,
\begin{align}\label{w1}
w_1 &=w+1+
	\sum_{j=0}^{k-1}\frac{\phi_j'(u)+F_j(u)\phi_j(u)-G_j(u)}{v^\frac{j}{k-1}} \\
	&\qquad +\0\left(\frac{1}{v^\frac{k}{k-1}},\frac{\phi_0'(u)}{v}\right)+
		\sum_{j=0}^{k-1}\0\left(
			\frac{(k-1-j)\phi_j(u)}{v^\frac{k}{k-1}},
			\frac{ju^\frac{k+1}{k}\phi_j'(u)}{v^\frac{k}{k-1}},
			\frac{\phi_j''\left(\hat{u}\right)}{v^\frac{k+j-1}{k-1}}\right) \notag
\end{align}
for $0\leq j\leq k-1$.  Now we want to show that we can find $\{\phi_j\}$ such that $\phi_j'+F_j\phi_j-G_j\equiv 0$ on $\Ouv$ and all of the $\0$-terms involving $\{\phi_j,\phi_j',\phi_j''\}$ are small enough. 

\begin{prop}\label{existsdiffeqsol}
If $R$ is large enough and $\delta,\theta$ are small enough, then there exist holomorphic functions $\{\phi_j\}_{0\leq j\leq k-1}$ on $\Ou$ such that
\begin{equation}\label{diffeq}
\phi_j'+F_j\phi_j-G_j\equiv 0
\end{equation}
on $\Ou$, where $F_j,G_j$ are defined using ~\eqref{FjGj}.  Furthermore, we have the bounds 
\begin{align}\label{OGphi}
G_j(u) &=\0\left(u^\frac{j-1}{k}\right),
	& G_j'(u)&=\0\left(u^\frac{j-1-k}{k}\right), && \notag \\
\phi_m(u) &=\0\left(u^\frac{m-1}{k}\right), 
	& \phi_m'(u)&=\0\left(u^\frac{m-1-k}{k}\right),
	& \phi_m''(u)&=\0\left(u^\frac{m-1-2k}{k}\right), \\
\phi_{k-1}(u)&=\0\left(u^\frac{2k-2}{k}\right),
	& \phi_{k-1}'(u)&=\0\left(u^\frac{k-2}{k}\right),
	& \phi_{k-1}''(u)&=\0\left(u^{-\frac{2}{k}}\right), \notag
\end{align}
on $\Ou$, where $0\leq m<k-1$.
\end{prop}

\begin{proof}
First we verify that for each $j$ and given $F_j,G_j$ then ~\eqref{diffeq} has a solution.  Given $\{\phi_l\}_{0\leq l< j}$, we can define $G_j$ as in ~\eqref{FjGj} and use this to define $\phi_j$.  Since $G_0$ is already defined, we can start this process. We follow the techniques used in Remark~\ref{Ouv} to choose $0\ll R_2<R_1<R$ and $0<\theta<\theta_1<\theta_2<\frac{\pi}{4}$ so that $F_j,G_j$ are holomorphic on $\Ou_{R_2,\theta_2}$, $\phi_j$ is holomorphic on $\Ou_{R_1,\theta_1}$, and the derivatives of $G_j,\phi_j$ are bounded on the subset $\Ou\subset\Ou_{R_1,\theta_1}\subset\Ou_{R_2,\theta_2}$, where these domains were defined in~\eqref{Ou}.   Fix $u_0$ such that $R_2<u_0<R_1$.  A solution to~\eqref{diffeq} is:
\begin{align}\label{diffeqsol}
\phi_j(u)&=e^{-\int_{u_0}^u F_j(\nu)d\nu}
\int_{u_0}^u G_j(\nu)e^{\int_{u_0}^\nu F_j(\zeta)d\zeta}d\nu,
\end{align}
where $u\in\Ou_{R_1,\theta_1}$ and the integral is taken along any simple, smooth curve between $u_0$ and $u$ contained in $\Ou_{R_2,\theta_1}\subset\Ou_{R_2,\theta_2}$.   Note that when $j=k-1$, $F_{k-1}\equiv 0$ so~\eqref{diffeqsol} simplifies to $\phi_{k-1}(u)=\int_{u_0}^uG_{k-1}(\nu)d\nu$.  Since there are only finitely many $j$ we can repeat this process to get $\{\phi_j\}_{0\leq j\leq k-1}$ holomorphic on $\Ou_{R_1,\theta_1}$ that satisfy the differential equation~\eqref{diffeq} and so that $\{G_j,G_j',\phi_j,\phi_j',\phi_j''\}_{0\leq j\leq k-1}$ are all defined and bounded on $\Ou$ as in~\eqref{OGphi}.  \\

Now we want to verify the orders in ~\eqref{OGphi}.  When $j=0$ and $u\in\Ou$ we know:
$$G_0(u)=-g_0(u)=\0\left(u^{-\frac{1}{k}}\right)\qquad\mbox{ and }\qquad
G_0'(u)=-g_0'(u)=\0\left(u^{-\frac{k+1}{k}}\right).$$
Suppose that the orders on $\{G_l,G_l',\phi_l,\phi_l',\phi_l''\}_{0\leq l<j}$ given in~\eqref{OGphi} hold for some $1\leq j\leq k-1$.  Then for $u\in\Ou$, 
$$G_j(u)=\0\left(g_j(u),\left\{\phi_l(u),\phi_l'(u)u^\frac{k+1}{k}\right\}_{0\leq l<j}\right)
	=\0\left(u^\frac{j-1}{k}\right) \qquad\mbox{ and }\qquad
G_j'(u)	=\0\left(u^\frac{j-1-k}{k}\right)$$
where $G_j'$ can be bounded on $\Ou$ using Cauchy estimates as describe in Remark~\ref{Ouv}.  Note that if $j=k-1$, then for $u\in\Ou$ and $n\in\N$,
$$\phi_{k-1}^{(n)}(u)= \frac{d^n}{du^n}\int_{u_0}^u G_{k-1}(\nu)d\nu=\0\left(u^\frac{k-2-(n-1)k}{k}\right).$$
It remains to show that for $0\leq j<k-1$, if the orders in~\eqref{OGphi} are satisfied by $\{G_l,G_l',\phi_l,\phi_l',\phi_l''\}_{0\leq l<j}$ and hence by $\{G_j,G_j'\}$, then they must also be satisfied by $\phi_j,\phi_j',\phi_j''$. Recall that $\{G_l\}_{0\leq l\leq j}$ are holomorphic on $\Ou_{R_2,\theta_2}$ and we want $\phi_j$ to be holomorphic on $\Ou_{R_1,\theta_1}$.  Given any $u\in\Ou_{R_1,\theta_1}$, define $c_u$ so that $\Arg (c_u)=\Arg (u)$ and $\re(c_u)=u_0$.  Then $c_u\in\Ou_{R_2,\theta_1}$.  Parametrize the line segment between $c_u$ and $u$ by $\gamma(t):=tu$, where $\frac{c_u}{u}\leq t\leq 1$.  By using integration by parts once in~\eqref{diffeqsol} and this parametrization we can express $\phi_j$ as:
\begin{align*}
\phi_j(u)
& =e^{-\int_{u_0}^u F_j(\nu)d\nu} \left(\frac{G_j(\nu)}{F_j(\nu)}e^{\int_{u_0}^\nu F_j(\zeta)d\zeta}\Bigg|_{u_0}^u-
\int_{u_0}^u \frac{G_j'(\nu)F_j(\nu)-G_j(\nu)F_j'(\nu)}{F_j(\nu)^2}e^{\int_{u_0}^\nu F_j(\zeta)d\zeta}d\nu\right) \\
&=\frac{G_j(u)}{F_j(u)}
+\0\left(e^{-\int_{u_0}^u F_j(\nu)d\nu}\right)
-e^{-\int_{u_0}^u F_j(\nu)d\nu}
\int_{c_u}^u \Gt_j(\nu)e^{\int_{u_0}^\nu F_j(\zeta)d\zeta}d\nu \\
&=\frac{G_j(u)}{F_j(u)}+\0\left(e^{-u}\right)-u\int_{\frac{c_u}{u}}^1 \Gt_j(tu)e^{-u\int_{t}^{1} F_j(\tau u)d\tau}dt
\end{align*}
where $u\in\Ou_{R_1,\theta_1}$, $\Gt_j(u):=\frac{G_j'(u)F_j(u)-G_j(u)F_j'(u)}{F_j(u)^2}=\0\left(u^\frac{j-1-k}{k}\right)$, $F_j(u)=\0(1)$, and we are integrating along $\gamma(t)$.  Assume $\Gt_j\not\equiv 0$ since otherwise there is nothing left to prove.  
\begin{align*}
\Bigg| \int_{\frac{c_u}{u}}^1 \Gt_j(tu)e^{-u\int_{t}^{1} F_j(\tau u)d\tau}dt \Bigg|
	&\leq \left(1-\frac{c_u}{u}\right) 
		\max_{\frac{c_u}{u}\leq t\leq 1}\left( \big| \Gt_j(tu)\big|
		e^{-\frac{k-1-j}{k-1}\re\left(u\int_t^1(1+g_0(\tau u))d\tau\right)}\right) \\
	&\leq C \max_{\frac{c_u}{u}\leq t\leq 1}
	(|u|t)^\frac{j-1-k}{k}e^{-\frac{k-1-j}{k-1}\frac{1-t}{2}|u|} 
\end{align*}
for some constant $C>0$.  When $t\neq 1$, the exponential term's exponent is a negative multiple of $|u|$ so the exponential term can be bounded above by a constant times an arbitrarily small power of $|u|$ whereas the other term remains bounded above by a constant.  When $t=1$, the integral is bounded above by $C|u|^\frac{j-1-k}{k}$.  Therefore, 
$$\phi_j(u)=\frac{G_j(u)}{F_j(u)}+\0\left(e^{-u}\right)+u\0\left(u^\frac{j-1-k}{k}\right)=\0\left(u^\frac{j-1}{k}\right)$$
on $\Ou_{R_1,\theta_1}$.  By shrinking the domain of $\phi_j',\phi_j''$ to $\Ou$, as discussed in Remark~\ref{Ouv}, we can use Cauchy's estimates and the order of $\phi_j$ to obtain the desired orders on the derivatives of $\phi_j$.  In particular, for any $u\in\Ou$ and $n\in\N$:
$$\phi_j^{(n)}(u)=\0\left(\frac{u^\frac{j-1}{k}}{u^n}\right)=\0\left(u^\frac{j-1-kn}{k}\right).$$
If we take $n=1,2$ we get the desired results.
\end{proof}

Let $(u,v)\in\Ouv$.  In the expression for $w_1$, we want to bound the terms $\phi_j''\left(\hat{u}\right)$ in terms of $u$ instead of $\hat{u}$.  Since $\hat{u}$ is on the line segment between $u$ and $u_1$, $\hat{u}\in\Ou$ and $|\hat{u}|>c|u|$ for some constant $c>0$.  Combining this with the previous result,
$$\phi_j''\left(\hat{u}\right)=\0\left(u^\frac{j-1-2k}{k}\right).$$

Now that we have bounds on $\phi_j$ and its derivatives, we can re-write $w_1$ from ~\eqref{w1} as:
\begin{equation*}
w_1=w+1+\epsilon_1(u,v)+\epsilon_2(u,v), 
\end{equation*}
where $\epsilon_1,\epsilon_2$ are functions of $u,v$ with the following orders:
\begin{align*}			
\epsilon_1(u,v) &=\0\left(\frac{\phi_0'(u)}{v},\frac{\phi_0''(\hat{u})}{v}\right)
	=\0\left(\frac{1}{u^\frac{k+1}{k}v}\right) \\
\epsilon_2(u,v) &=\0\left(\frac{1}{v^\frac{k}{k-1}}\right)+
	\sum_{j=0}^{k-1}\0\left(
		\frac{(k-1-j)\phi_j(u)}{v^\frac{k}{k-1}},
		\frac{ju^\frac{k+1}{k}\phi_j'(u)}{v^\frac{k}{k-1}},
		\frac{j\phi_j''(\hat{u})}{v^\frac{k+j-1}{k-1}}\right)
	 =\0\left(\frac{u^\frac{2k-1}{k}}{v^\frac{k}{k-1}}\right).
\end{align*}

\begin{prop}For any $(u,v)\in\Ouv$ and $w=w(u,v)$, we have 
$$\lim_{n\to\infty} \frac{u_n}{\log n}=1\qquad\mbox{ and }\qquad\lim_{n\to\infty}\frac{w_n}{n}=1.$$
\end{prop}
\begin{proof}The sequence $\left(\frac{w_n}{n}\right)$ converges to $1$ since, for some constants $C,l>0$, 	
$$\lim_{n\to\infty}\frac{1}{n}\sum_{j=0}^{n-1}\left(\epsilon_1(u_j,v_j)+\epsilon_2(u_j,v_j)\right)
\leq \lim_{n\to\infty}\frac{1}{n}\sum_{j=l}^{n-1}\left(
\frac{C}{(\log j)^\frac{k+1}{k}j}+\frac{C(\log j)^\frac{2k-1}{k}}{j^\frac{k+1}{k}}\right)=0,$$
where we are using the bounds on $u_n,v_n$ from Proposition~\ref{uvBound}, and
$$\lim_{n\to\infty}\frac{w_n}{n} 
	=\lim_{n\to\infty}\frac{1}{n}\left[w+n+\sum_{j=0}^{n-1}
		\left(\epsilon_1(u_j,v_j)+\epsilon_2(u_j,v_j)\right)\right]
	=1.$$
In order to show $\lim_{n\to\infty}\frac{u_n}{\log n}=1$, we first need to replace the $v_j$ terms in $u_n$ with $w_j$ terms:
$$\frac{1}{v} 
	=\frac{1}{w}\left(1+\sum_{j=0}^{k-1}\frac{\phi_j(u)}{v^\frac{j}{k-1}}\right)
		=\frac{1}{w}\left( 1+\phi_0(u)+\0\left(\frac{1}{v^\frac{1}{k-1}}\right)\right).$$
We use the bounds on $u_n,v_n$ from Lemma~\ref{uvBound} to show that the sequence $\left(\frac{u_n}{\log n}\right)$ converges to $1$.  
\begin{align*}
\lim_{n\to\infty}\frac{u_n}{\log n}
	&=\lim_{n\to\infty}\frac{1}{\log n}\left[u+\sum_{j=0}^{n-1}\left(\frac{1}{v_j}+\frac{u_j^\frac{k+1}{k}}{v_j^\frac{k}{k-1}}\h(u_j,v_j)+\0\left(\frac{1}{v_j^2}\right) \right)\right] \\
	&=\lim_{n\to\infty}\frac{1}{\log n}\sum_{j=0}^{n-1}\frac{1}{v_j}
	=\lim_{n\to\infty}\frac{1}{\log n}\sum_{j=0}^{n-1}
		\frac{1}{w_j}\left( 1+\phi_0(u_j)\right) \\	
	&=\lim_{n\to\infty}\frac{1}{\log n}\sum_{j=m}^{n-1}\frac{1}{j}+
		\lim_{n\to\infty}\frac{1}{\log n}\sum_{j=m}^{n-1}\frac{\phi_0(u_j)}{j}
	=1,	
\end{align*}
where we recall that $\h(u,v)=\0(1)$ and we choose $m<n-1$ large enough that we can replace $w_j$ by $j$ for $j\geq m$.  We arrive at the final equality using:
$$\log n-\log m=\int_{m-1}^{n}\frac{dx}{x}\leq \sum_{j=m}^{n-1}\frac{1}{j}\leq \int_{m-1}^{n-1}\frac{dx}{x}=\log(n-1)-\log(m-1),$$
the fact that $\phi_0(u)=\0\left(u^{-\frac{1}{k}}\right)$, and
$$\frac{1}{\log n}\sum_{j=m}^{n-1}\frac{1}{j(\log j)^\frac{1}{k}}
\leq\frac{1}{(\log n)^\frac{1}{2k}}\sum_{j=m}^{n-1}\frac{1}{j(\log j)^\frac{2k+1}{2k}}
\leq\frac{1}{(\log n)^\frac{1}{2k}}\int_{m-1}^{n-1}\frac{dx}{x(\log x)^\frac{2k+1}{2k}}
=\0\left(\frac{1}{(\log n)^\frac{1}{2k}}\right).$$
\end{proof}

\begin{prop}
\label{Fatou1}
The sequence $\left(w_n-n\right)_{n=0}^\infty$ converges uniformly on compact subsets of $\Ouv$.  Let $\omega$ be the holomorphic limit function.  Then $\omega(u_1,v_1)=\omega(u,v)+1.$
\end{prop}

The function $\omega$ is usually referred to as a Fatou coordinate. 
\begin{proof}
For some $C>0$ and any $n,m\in\N$ with $n>m$,
\begin{align*}
&\big|(w_n-n)-(w_m-m)\big|
	=\Bigg| \sum_{l=m}^{n-1}\left[\epsilon_1(u_l,v_l)+\epsilon_2(u_l,v_l)\right] \Bigg|
	\leq \sum_{l=m}^{n-1}\left(
		\frac{C}{|u_l|^\frac{k+1}{k}|v_l|}+\frac{C|u_l|^\frac{2k-1}{k}}{|v_l|^\frac{k}{k-1}}\right) \\
&		\qquad\leq\sum_{l=m}^{n-1}\left[
			\frac{\frac{2C}{\re(v)}}{\left(\re(u)+\frac{1}{6}\log\left(1+\frac{l}{\re(v)}\right)\right)^\frac{k+1}{k}\left(1+\frac{l}{\re(v)}\right)}
			+\frac{C\left(\re(u)+3\log\left(1+\frac{l}{\re(v)}\right) \right)^\frac{2k-1}{k}}{\left(\re(v)+\frac{l}{2}\right)^\frac{k}{k-1}} \right]\\
&		\qquad\leq \left[-12Ck\left(\re(u)+\frac{1}{6}\log\left(1+\frac{x}{\re(v)}\right)\right)^{-\frac{1}{k}}\right]\Bigg|_{m-1}^{n-1}
			+\frac{C\left(\re(u)\right)^\frac{2k-1}{k}}{\left(\re(v)\right)^\frac{k}{k-1}}
				\int_{m-1}^{n-1}\frac{dx}{\left(1+\frac{x}{2\re(v)}\right)^{\frac{k+1}{k}}}\\
&		\qquad\leq 12Ck\left(\re(u)+\frac{1}{6}\log\left(1+\frac{m-1}{\re(v)}\right)\right)^{-\frac{1}{k}}+2Ck\frac{\re(u)^\frac{2k-1}{k}}{\re(v)^\frac{1}{k-1}}\left(1+\frac{m-1}{2\re(v)}\right)^{-\frac{1}{k}}
\end{align*}
where we used bounds from Lemma~\ref{uvBound}.  For any compact $K\subset\Ouv$, $\exists S$ such that $|u|,|v|<S, \forall (u,v)\in K$.  So
$$|(w_n-n)-(w_m-m)|
<\frac{12Ck}{\left(R+\frac{1}{6}\log\left(1+\frac{m-1}{S}\right)\right)^\frac{1}{k}}+\frac{2CkS^\frac{2k-1}{k}R^{-\frac{1}{k-1}}}{\left(1+\frac{m-1}{2S}\right)^\frac{1}{k}}.$$
Therefore $\forall\epsilon>0, \exists M\in\N$ such that $\forall n,m>M$ and $\forall(u,v)\in K$, $|(w_n-n)-(w_m-m)|<\epsilon$.  Hence, the sequence of holomorphic functions $(w_n-n)$ converges uniformly on compact subsets of $\Ouv$ to a holomorphic limit function:
$$\omega(u,v)=\lim_{n\to\infty}(w_n-n)=\lim_{n\to\infty}\sum_{j=0}^{n-1} (w_{j+1}-w_j-1)+w.$$
Finally we show that $\omega\circ\fa=\omega+1$ on $\Ouv$:
\begin{align*}
\omega(u_1,v_1) &=\sum_{j=0}^\infty \left(w_{j+2}-w_{j+1}-1\right)+w_1 \\
	&=\sum_{j=0}^\infty \left(w_{j+1}-w_{j}-1\right)+w+1 \\
	&=\omega(u,v)+1
\end{align*}
\end{proof}
Define the function $\eta$ as follows for any $(u,v)\in\Ouv$:
\begin{equation}\label{eta}
\eta(u,v):=\omega(u,v)-w=\lim_{n\to\infty}\left((w_n-n)-(w-0)\right)
\end{equation}
From the proof of Proposition~\ref{Fatou1}, we can bound $\eta(u,v)$ for any $(u,v)\in\Ouv$:
\begin{equation}\label{etaapprox}
|\eta(u,v)|\leq 12Ck\left(\re(u)-1\right)^{-\frac{1}{k}}+4Ck \re(u)^\frac{2k-1}{k}\re(v)^{-\frac{1}{k-1}}.
\end{equation}
Now we will show that the following map is a coordinate change from coordinates $(u,v)$ to $(u,\omega)$:
$$\psi_1(u,v):=(u,\omega(u,v)).$$
First we need to define a couple of domains that contain $\Ouv$.  We choose appropriate constants $0\ll R_2<R_1<R,0<2\delta<\delta_2\ll 1,\mbox{ and }0<3\theta<\frac{\pi}{4}$ so that 
$$\Ouv\subsetneq\Omega_1:=\Ouv_{R_1,2\theta,2\delta}\subsetneq\Omega_2:=\Ouv_{R_2,3\theta,\delta_2}$$
and all of these domains have satisfied the properties shown thus far for $\Ouv$. 
\begin{prop} \label{psi1}
Let
\begin{align*}
\Ouw_2 &:=\left\{(u,\omega) \ \bigg| \ \re(u)>R, |u|^\frac{(k-1)(k+1)}{k}<\delta|\omega|, |\Arg (u)|<\theta, |\Arg (\omega)|<\frac{k-1}{k}\theta \right\} \\
\Ouv_2 &:=\Omega_1\cap\psi_1^{-1}\left(\Ouw_2\right)=\Ouv_{R_1,2\theta,2\delta}\cap\psi_1^{-1}\left(\Ouw_2\right)
\end{align*}
Then $\psi_1:\Ouv_2\to\Ouw_2$ is a biholomorphism.\end{prop}

\begin{proof}  
First of all, $\psi_1$ is holomorphic on $\Omega_2$ since it is in each component.  Now we want to find a domain on which $\psi_1$ is injective. 
For any $(u,v),(\hat{u},\hat{v})\in\Ouv$,
$$\psi_1(u,v)=\psi_1(\hat{u},\hat{v})\Leftrightarrow
u=\hat{u}\mbox{ and }\omega(u,v)=\omega(\hat{u},\hat{v}).$$
Fix $u_0\in\Ou$ and let $\omega_{u_0}(v):=\omega(u_0,v)$.  Define
\begin{align*}
\Omega_{2,u_0}&:=\{v\in\C \ | \ (u_0,v)\in\Omega_2\} \\
\Omega_{1,u_0}&:=\{v\in\C \ | \ (u_0,v)\in\Omega_1\} \\
\Omega_{u_0}&:=\left\{v\in\C \ | \ (u_0,v)\in\Ouv\right\} 
\end{align*}
Then $\omega_{u_0}$ is holomorphic on $\Omega_{2,u_0}$.  Fix any $y\in\Omega_{u_0}$.  For any $v\in\Omega_{2,u_0}$, let 
$$g(v):=\omega_{u_0}(v)-y\qquad\mbox{ and }\qquad h(v):=v-y.$$
Then $g,h$ are holomorphic on $\Omega_{2,u_0}$.  Let $\gamma$ be the curve that is the boundary of the region:
$$\Omega_{1,u_0}\cap \{v\in\C \ | \ |v|<2|y|\}=\left\{v\in\C \ \bigg| \ (2\delta)^{-1}|u_0|^\frac{k^2-1}{k}<|v|<2|y|\mbox{ and }|\Arg(v)|<2\theta\right\}.$$
The point $y$ lies inside the region bounded by the curve $\gamma$.  Note that:
$$|g(v)-h(v)| 
	=\Bigg|\sum_{j=0}^{k-1}\frac{\phi_j(u_0)}{v^{\frac{j}{k-1}}}+\frac{\eta(u_0,v)}{v}\Bigg||v|
	<\frac{c|v|}{|u_0|^\frac{1}{k}}$$
for some constant $c>0$.  We want to show that on $\gamma$:
$$|g(v)-h(v)|<|h(v)|,\mbox{ which we can prove by showing that }\frac{c}{|u_0|^\frac{1}{k}}<\frac{|h(v)|}{|v|}=\frac{|v-y|}{|v|}.$$
We bound $\frac{|h(v)|}{|v|}$ from below on the segments of $\gamma$: 
$$\mbox{(1) }|v|=(2\delta)^{-1}|u_0|^\frac{k^2-1}{k},\qquad\mbox{ (2) }|v|=2|y|,\qquad\mbox{ (3) }|\Arg(v)|=2\theta.$$  
On the first segment of $\gamma$:
$$|h(v)|=|v-y|\geq |y|-|v|>\left(\delta^{-1}-(2\delta)^{-1}\right)|u_0|^\frac{k^2-1}{k}=(2\delta)^{-1}|u_0|^\frac{k^2-1}{k}=|v|$$

On the second segment of $\gamma$:
$$|h(v)|=|v-y|\geq|v|-|y|=|y|=\frac{1}{2}|v|.$$

On the third segment of $\gamma$, $|\Arg(v)|=2\theta$.  Fix any $v$ on this segment and without loss of generality assume $\Arg(v)=2\theta$.  The distance $|v-y|$ is greater than the shortest distance from $v$ to the line of angle $\theta$ from the origin, 
therefore
$$|h(v)|=|v-y|\geq|v|\sin(\theta).$$

Hence on $\gamma$ we have:
$$\frac{|h(v)|}{|v|}=\frac{|v-y|}{|v|}\geq\min\bigg\{1,\frac{1}{2},\sin\theta\bigg\}=\sin\theta.$$
By requiring that $R>\left(\frac{c}{\sin\theta}\right)^k$, we get
$$|g(v)-h(v)|<|h(v)|=|v-y|\mbox{ on }\gamma.$$  
From this inequality, we know that neither $g(v)$ nor $h(v)$ has a zero on $\gamma$.  Since the region $\Omega_{2,u_o}$ contains the closed curve $\gamma$ and $g,h$ are holomorphic on $\Omega_{2,u_o}$ with no zeros on $\gamma$, we can extend the closed, connected set bounded by $\gamma$ to a region slightly larger that contains no extra zeros of $g$ or $h$. By Rouch\'{e}'s theorem, $g$ and $h$ have the same number of zeros on this region.  Since $h(v)=v-y$ has exactly one zero in this region, $g(v)$ must as well.  Note that if $v\in\Omega_{1,u_o}$ and $|v|\geq 2|y|$, then it is not possible for $\omega(u_o,v)=y$ (we can see this from the calculation for $y$ in terms of $v$).  Therefore $g$ is injective on $\Omega_{1,u_0}$ for any $u_0\in\Ou$ and so $\psi_1$ is injective on $\Omega_1$.  Furthermore, $\forall (u_o,y)\in\Ouw_2$ we know that $(u_o,y)\in\Ouv$ and so there exists a unique element $(u_o,v)\in\Omega_1$ such that $\psi_1(u_o,v)=(u_o,y)$.  Therefore, $\psi_1:\Ouv_2 \to\Ouw_2$ is a biholomorphism.
\end{proof}

\section{Conjugacy to Translation}
In this section we make a coordinate change so that composition with $f$ acts as the identity map on the first component.  We can re-write \eqref{u} and \eqref{eta} for any $(u,v)\in\Ouv$ as:
\begin{align}
\omega &= w+\eta(u,v) = v\left(1+\sum_{l=0}^{k-1}\frac{\phi_l(u)}{v^\frac{l}{k-1}}+\frac{\eta(u,v)}{v}\right) \label{omega} \\
\frac{1}{v} &= \frac{1}{\omega}\left(1+\sum_{l=0}^{k-1}\frac{\phi_l(u)}{v^\frac{l}{k-1}}+\frac{\eta(u,v)}{v}\right) =\0\left(\frac{1}{\omega}\right) \notag \\
u_1 &= u+\frac{1}{v}+\frac{u^\frac{k+1}{k}}{v^\frac{k}{k-1}}\h(u,v)+\0\left(\frac{1}{v^2}\right) \notag
\end{align}
where $h=\0(1)$.  Let $\fb:=\psi_1\circ\fa\circ\psi_1^{-1}$.  Then for any $(u,\omega)\in\Ouw_2$, we can express $\fb(u,\omega)=(u_1,\omega_1)$ as:
\begin{align}
u_1 &=u+\frac{1}{\omega}+\frac{\phi_0(u)}{\omega}+\sum_{l=1}^{k-1}\0\left(\frac{\phi_l(u)}{\omega^\frac{l+k-1}{k-1}}\right)+\0\left(\frac{u^\frac{k+1}{k}}{\omega^\frac{k}{k-1}},\frac{1}{\omega^2}\right) \label{u(w)}\\
	&=u+\frac{1}{\omega}+\frac{\phi_0(u)}{\omega}+\0\left(\frac{u^\frac{k+1}{k}}{\omega^\frac{k}{k-1}}\right) 	\notag \\
\omega_1 &= \omega+1 \label{w},
\end{align}
where we use~\eqref{OGphi} and ~\eqref{etaapprox} to bound $\{\phi_l\}$ and $\eta$.  By employing the same techniques as in the proof of Lemma~\ref{Ouvinvariant}, we can show that if $(u,\omega)\in\Ouw_2$, then $(u_1,\omega_1)\in\Ouw_2$.  Since $\psi$ is surjective, if $(u,\omega)\in\Ouw_2$, then $\exists (u,v)\in\Ouv_2$ such that $\psi_1(u,v)=(u,\omega)$ and $\psi_1(u_1,v_1)=(u_1,\omega_1)$.  We can use \eqref{omega} to roughly bound $\re v$ by $\re\omega$: $\frac{\re\omega}{2}<\re v<2\re\omega$.  Combining this with \eqref{uineq}, we derive the inequality:
\begin{equation}\label{uwbound}
3\log\left(1+\frac{2n}{\re(\omega)}\right)\geq
\re(u_n)-\re(u) \geq 
\frac{1}{6}\log\left(1+\frac{n}{2\re(\omega)}\right).
\end{equation}
\begin{prop}\label{Fatou2}
Fix some $u_0\in\pi_1\left(\Ouw_2\right)$ and let $1\leq l\leq k$.  Define:
\begin{equation} \label{alpha}
\alpha_l(u):=\int_{u_0}^u \phi_0^l(\zeta)d\zeta,\qquad \alpha(u):=\sum_{l=1}^k (-1)^l\alpha_l(u),\quad\mbox{ and }\quad t_n:=u_n-\log \omega_n+\alpha(u_n)
\end{equation}
for any $(u,\omega)\in\Ouw_2$ and $n\in\N$.  Then the sequence $\left(t_n\right)_{n=0}^\infty$ converges uniformly on compact subsets of $\Ouw_2$.  Let $\tau$ be the holomorphic limit function of the sequence.  Then $\tau(u_1,\omega_1)=\tau(u,\omega).$
\end{prop}

\begin{proof}For any $(u,\omega)\in\Ouw_2$ and $j\in\N$, we use Taylor series expansion as in~\eqref{phij} to get:
\begin{align*}
\int_{u_j}^{u_{j+1}} \phi_0^l(\zeta)d\zeta 
	&=\phi_0^l(u_j)(u_{j+1}-u_j)+\frac{(u_{j+1}-u_j)^2}{2}l\phi_0^{l-1}\left(\hat{u}\right)\phi_0'\left(\hat{u}\right) \\
	&=\phi_0^l(u_j)\left(\frac{1}{\omega_j}+\frac{\phi_0(u_j)}{\omega_j}\right)+\0\left(\frac{u_j^\frac{k+1}{k}}{\omega_j^\frac{k}{k-1}}\right)
\end{align*}
where $\hat{u}$ is some point on the line between $u_j$ and $u_{j+1}$, which may vary depending on $j,l$.  Since $\hat{u}$ is on the line segment between $u_j$ and $u_{j+1}$, $\hat{u}\in\pi_1\left(\Ouv_2\right)$, $\phi_0(\hat{u})$ is defined, and $C|u_j|>|\hat{u}|>c|u_j|$ for some constants $C>c>0$.  We used this along with~\eqref{OGphi} to bound the terms involving $\hat{u}$.  For any $n,m\in\N$ with $n>m$,
\begin{align*}
u_n-u_m
	&=\sum_{j=m}^{n-1}\left(\frac{1}{\omega_j}+\frac{\phi_0(u_j)}{\omega_j}+\0\left(\frac{u_j^\frac{k+1}{k}}{\omega_j^\frac{k}{k-1}}\right)\right) \\
\alpha(u_n)-\alpha(u_m)
	&=\sum_{j=m}^{n-1} \left(\alpha(u_{j+1})-\alpha(u_j)\right)
	=\sum_{j=m}^{n-1} \sum_{l=1}^k (-1)^l \int_{u_j}^{u_{j+1}}\phi_0^l(\zeta)d\zeta \\
	&=\sum_{j=m}^{n-1}\sum_{l=1}^k (-1)^l \left(\phi_0^l(u_j)\left(\frac{1}{\omega_j}+\frac{\phi_0(u_j)}{\omega_j}\right)+\0\left(\frac{u_j^\frac{k+1}{k}}{\omega_j^\frac{k}{k-1}}\right) \right)
\end{align*}
Now we  bound $|t_n-t_m|$ using the previous equations:
\begin{align*}
\big|t_{n}-t_m\big|
	&=\Big| (u_n-u_m)-(\log \omega_n-\log \omega_m)+(\alpha(u_n)-\alpha(u_m)) \Big| \\
	&\leq \Bigg| \sum_{j=m}^{n-1}\left(
		\sum_{l=0}^k (-1)^l \phi_0^l(u_j)\left(\frac{1}{\omega_j}+\frac{\phi_0(u_j)}{\omega_j}\right)
		-\log\left(\frac{\omega_{j+1}}{\omega_j}\right)
		 \right) \Bigg|+c\sum_{j=m}^{n-1}\frac{|u_j|^\frac{k+1}{k}}{|\omega_j|^\frac{k}{k-1}} \\
	&\leq \Bigg| \sum_{j=m}^{n-1}\left(
		\frac{1}{\omega_j}
		-\log\left(1+\frac{1}{\omega_j}\right) \right) \Bigg|
	+c\sum_{j=m}^{n-1}\left(
		\frac{|u_j|^\frac{k+1}{k}}{|\omega_j|^\frac{k}{k-1}}
		+\frac{1}{|u_j|^\frac{k+1}{k}|\omega_j|}\right) \\
	&\leq c\sum_{j=m}^{n-1}\left(
		\frac{|u_j|^\frac{k+1}{k}}{|\omega_j|^\frac{k}{k-1}}
		+\frac{2}{|u_j|^\frac{k+1}{k}|\omega_j|}\right)  \\
	&\leq \frac{2c\re(u)^\frac{k+1}{k}}{\re(\omega)^\frac{k}{k-1}}\sum_{j=m}^{n-1}
		\frac{\left(1+\frac{3}{\re(u)}\log\left(1+\frac{2j}{\re(\omega)}\right)\right)^\frac{k+1}{k}}{\left(1+\frac{j}{\re(\omega)}\right)^\frac{k}{k-1}} \\
&		\qquad+2c\sum_{j=m}^{n-1}\frac{1}{\left(\re(u)+\frac{1}{6}\log\left(1+\frac{j}{2\re(\omega)}\right)\right)^\frac{k+1}{k}\left(1+\frac{j}{2\re(\omega)}\right)\re(\omega)}  \\
	&\leq \frac{2c\re(u)^\frac{k+1}{k}}{\re(\omega)^\frac{k}{k-1}}\int_{m-1}^{n-1}\left(1+\frac{x}{\re(\omega)}\right)^{-\frac{k+1}{k}}dx
		-24k c\left(\re(u)+\frac{1}{6}\log\left(1+\frac{x}{2\re(\omega)}\right)\right)^{-\frac{1}{k}}\Bigg|_{m-1}^{n-1} \\
	&\leq 2kc\frac{\re(u)^\frac{k+1}{k}}{\re(\omega)^\frac{1}{k-1}}\left(1+\frac{m-1}{\re(\omega)}\right)^{-\frac{1}{k}}+
			24k c\left(\re(u)+\frac{1}{6}\log\left(1+\frac{m-1}{2\re(\omega)}\right)\right)^{-\frac{1}{k}}
\end{align*}
for some constant $c>1$ independent of $(u,\omega)$. For any compact $K\subset\Ouw_2, \exists S$ such that $|u|,|v|<S, \forall (u,v)\in K$.  Then
$$\big|t_{n}-t_m\big|<
 \frac{2kcS^\frac{k+1}{k}}{R^\frac{1}{k-1}}\left(1+\frac{m-1}{S}\right)^{-\frac{1}{k}}+24k c\left(R+\frac{1}{6}\log\left(1+\frac{m-1}{2S}\right)\right)^{-\frac{1}{k}}.$$

Therefore $\forall\epsilon>0, \exists M\in\N$ such that $\forall n,m>M$ and $\forall(u,v)\in K$, $|t_n-t_m|<\epsilon$.  Hence, the sequence of holomorphic functions $(t_n)$ converges uniformly on compact subsets of $\Ouw_2$ to the limit function $\tau$, which also must be holomorphic on $\Ouw_2$.  So
$$\tau(u,\omega)=\lim_{n\to\infty}t_n=\sum_{j=0}^\infty (t_{j+1}-t_j)+t.$$

Finally we show that $\tau\circ\fb=\tau$ on $\Ouw_2$:
$$\tau(u_1,\omega_1) =
	\sum_{j=0}^\infty (t_{j+2}-t_{j+1} )+t_1-\left(t-t\right) 
	=\sum_{j=0}^\infty (t_{j+1}-t_{j} )+t 
	=\tau(u,\omega)$$
\end{proof}

Now we show that the following  map is a coordinate change:
$$\psi_2(u,\omega):=(\tau(u,\omega),\omega).$$
Define the function $\mu$ as follows for any $(u,\omega)\in\Ouw_2$:
\begin{equation}\label{mu}
\mu(u,\omega):=\tau(u,\omega)-u+\log \omega-\alpha(u)=\lim_{n\to\infty}\left(t_n-t\right).
\end{equation}
We can bound $\mu(u,\omega)$ for any $(u,\omega)\in\Ouw_2$ using the proof of Proposition~\ref{Fatou2} and $m=0$:
\begin{equation}\label{mubound}
|\mu(u,\omega)|\leq 4kc\frac{\re(u)^\frac{k+1}{k}}{\re(\omega)^\frac{1}{k-1}}+\frac{24kc}{(\re(u)-1)^\frac{1}{k}}.
\end{equation}

Let
\begin{equation}\label{Ouwdef}
\Ouw_{R',\delta',\theta'}:=
\left\{(u,\omega) \ \bigg| \ \re(u)>R',|u|^\frac{(k-1)(k+1)}{k}<\delta'|\omega|,|\Arg(u)|<\theta',|\Arg(\omega)|<\frac{k-1}{k}\theta' \right\}
\end{equation}
for any $R',\delta',\theta'$.  To simplify notation, replace $(R,\delta,\theta)$ by $(R_2,\delta_2,\theta_2)$ in the preceding work so that $\Ouw_2=\Ouw_{R_2,\delta_2,\theta_2}$.  Given $(R_2,\delta_2,\theta_2)$, we choose appropriate constants $R_2<R_1<R,0<\delta<\delta_1<\delta_2,\mbox{ and }0<\theta<\theta_1<\theta_2 $ so that 
$$\Omega_0:=\Ouw_{R,\delta,\theta}\subsetneq\Omega_1:=\Ouw_{R_1,\delta_1,\theta_1}\subsetneq\Ouw_2$$
and all of these domains have satisfied the properties shown thus far for $\Ouw_2$.

\begin{prop}
\label{psi2}
Let
\begin{align*}
\Otwh &:=\left\{(\tau-\log(\omega),\omega) \ \bigg| \ \re(\tau)>R,|\tau|<\delta|\omega|^\frac{k}{(k-1)(k+1)},|\Arg(\tau)|<\theta, |\Arg(\omega)|<\frac{k-1}{k}\theta\right\} \\
\Ouwh &:=\Omega_1\cap\psi_2^{-1}\left(\Otwh\right) 
\end{align*}
Then $\psi_2:\Ouwh\to\Otwh$ is a biholomorphism onto its image.
\end{prop}

\begin{proof}
We use a similar strategy as in the proof of Proposition~\ref{psi1}.  $\psi_2$ is holomorphic on $\Ouw_2$ since it is in each component.  For any $(u,w),(\hat{u},\hat{w})\in\Ouw_2$,
$$\psi_2(u,w)=\psi_2(\hat{u},\hat{w})\Leftrightarrow \tau(u,w)=\tau(\hat{u},\hat{w})\mbox{ and }w=\hat{w}.$$
Fix $w_0\in\pi_2\left(\Omega_0\right)$ and let $\tau_{w_0}(u):=\tau(u,w_0)$.  Define
\begin{align*}
\Omega_{2,w_0} &:= \left\{ u\in\C \ | \ (u,w_0)\in \Ouw_2 \right\} \\
\Omega_{1,w_0} &:= \{ u\in\C \ | \ (u,w_0)\in \Omega_1 \} \\
\Omega_{0,w_0} &:= \{ u\in\C \ | \ (u,w_0)\in \Omega_0 \} 
\end{align*}
Then $\tau_{w_0}$ is holomorphic on $\Omega_{2,w_0}$.  Fix any $y\in\Omega_{0,w_0}$.  Let
$$g(u):=\tau_{w_0}(u)+\log(w_0)-y\qquad\mbox{ and }\qquad h(u):=u-y$$
Then $g,h$ are holomorphic on $\Omega_{2,w_0}$.  Let $\gamma$ be the curve that is the boundary of the region:
$$ \Omega_{1,w_0} = \left\{ u\in\C \ \Big| \re(u)>R_1, |u|<\left(\delta_1 |w_0|\right)^\frac{k}{(k-1)(k+1)}, |\Arg(u)|<\theta_1 \right\}.$$
The point $y$ lies inside the curve $\gamma$.  Using~\eqref{OGphi}, we bound $\alpha$:
$$\alpha(u)=\sum_{l=1}^{k} (-1)^l \int_{u_0}^u \phi_0^l(\zeta)d\zeta=\0\left(\int_{u_0}^u \phi_0(\zeta)d\zeta\right)=\0\left(u^\frac{k-1}{k}\right).$$
Then for $u\in\Omega_{2,w_0}$, using this bound and the one on $\mu$ given in~\eqref{mubound}, we get:
$$|g(u)-h(u)|=|\mu(u,w_0)+\alpha(u)|<c |u|^{1-\frac{1}{k}}<|u|^{1-\frac{1}{k^2-1}},$$
for some $c>0$ and $R_2$ large enough.  We want to show that on $\gamma$:
$$|g(u)-h(u)|<|h(u)|.$$
We bound $|h(u)|$ from below on the segments of $\gamma$:
$$\mbox{(1) }\re(u)=R_1,\qquad \mbox{ (2) }|u|=\left(\delta_1 |w_0|\right)^\frac{k}{(k-1)(k+1)},\qquad\mbox{ (3) } |\Arg(u)|=\theta_1.$$
On the first segment of $\gamma$:
$$|h(u)|\geq\left(\frac{|y|}{|u|}-1\right)|u|>\left(\frac{R_0}{\sqrt{2}R_1}-1\right)|u| \geq |u|^{1-\frac{1}{k^2-1}},$$
where the last inequality follows if we assume that $R_0 \geq \sqrt{2}R_1\left(1+R_1^{-\frac{1}{k^2-1}}\right).$  \\
On the second segment of $\gamma$:
$$|h(u)|\geq |u|-|y|>|u|\left(1-\left(\frac{\delta_0}{\delta_1}\right)^\frac{k}{k^2-1}\right)>|u|^{1-\frac{1}{k^2-1}},$$
where the last inequality follows if we assume that $R_1\geq\left(1-\left(\frac{\delta_0}{\delta_1}\right)^\frac{k}{k^2-1}\right)^{-(k^2-1)}$. \\
On the third segment of $\gamma$:
$$|h(u)|=|u-y|>|u| \sin(\theta_1-\theta_0) \geq |u|^{1-\frac{1}{k^2-1}},$$
where the last inequality follows if we assume that $R_1\geq \left(\sin(\theta_1-\theta_0)\right)^{-(k^2-1)}$. \\
Given $\{\theta_j,\delta_j\}_{0\leq j\leq 2}$, we can choose $\{R_j\}_{0\leq j\leq 2}$ large enough.  Hence
$$|h(u)|>|g(u)-h(u)|\mbox{ on }\gamma.$$
From this inequality, we know that neither $g$ nor $h$ has a zero on $\gamma$.  Since the region $\Omega_{2,w_0}$ contains the closed curve $\gamma$ and both $g$ and $h$ are holomorphic on $\Omega_{2,w_0}$ with no zeros on $\gamma$, we can extend the closed, connected set bounded by $\gamma$ to a region slightly larger that contains no extra zeros of $g$ or $h$.  By Rouch\'{e}'s theorem, $g$ and $h$ have the same number of zeros on this region; $h$ has exactly one zero in this region, hence so does $g$.  So $\forall (y,w_0)\in\Omega_0$, there is a unique element $u\in\Omega_{1,w_0}$ such that $\tau(u,w_0)=y-\log(w_0)$.  Consequently, for any $(y-\log(w_0),w_0)\in\Otwh$ there is a unique $u\in\Omega_{1,w_0}$ such that $\psi_2(u,w_0)=(\tau(u,w_0),w_0)=(y-\log(w_0),w_0)$.  
Therefore $\psi_2:\Ouwh\to\Otwh$ is a biholomorphism.
\end{proof}
Let 
\begin{align}
\Ouvh &:=\psi_1^{-1}\left(\Ouwh\right),
&\Oxyh &:=\psi_0^{-1}\left(\Ouvh\right), 
&\Omega&:=l\left(\Oxyh\right)
&&\label{Omega} \\
\Psi&:=\psi_2\circ\psi_1\circ\psi_0,&\mbox{ and }&&\fc&:=\Psi\circ\fo\circ\Psi^{-1}.\label{Psi}
\end{align}
\begin{lem}$\fc\left(\Otwh\right)\subseteq\Otwh$ and $\fo\left(\Oxyh\right)\subseteq\Oxyh$.\end{lem}
\begin{proof}For any $(\hat{\tau},\omega)\in\Otwh$, let $\tau:=\hat{\tau}+\log(\omega)$.  Then $\fc\left(\hat{\tau},\omega\right)=(\hat{\tau},\omega+1)=(\tau-\log(\omega),\omega+1)$.  Let
$$t:=\hat{\tau}+\log(\omega+1)=\tau+\log(1+\omega^{-1}).$$
To show that $(\hat{\tau},\omega+1)=(t-\log(\omega+1),\omega+1)\in\Otwh$, we need:
$$\mbox{(1) }\re(t)>R,\qquad\mbox{ (2) }|t| <(\delta|\omega+1|)^\frac{k}{(k-1)(k+1)},\qquad\mbox{ (3) }|\Arg(t)|<\theta.$$
First of all,
$$\re(t)=\re(\tau)+\log\big|1+\omega^{-1}\big|>R.$$
Secondly,
$$|t|=|\tau| \big|1+\tau^{-1}\log\left(1+\omega^{-1}\right)\big|<(\delta|\omega|)^\frac{k}{(k-1)(k+1)}  \big|1+\tau^{-1}\log\left(1+\omega^{-1}\right)\big|<(\delta|\omega+1|)^\frac{k}{(k-1)(k+1)},$$
where the last inequality follows from the Taylor series expansions: 
\begin{align*}
1+\tau^{-1}\log\left(1+\omega^{-1}\right)=1+\tau^{-1}\omega^{-1}+\0\left(\tau^{-1}\omega^{-2}\right) \\
\left(1+\omega^{-1}\right)^\frac{k}{(k-1)(k+1)}=1+\frac{k}{k^2-1}\omega^{-1}+\0\left(\omega^{-2}\right)
\end{align*}
Finally, 
$$|\Arg(t)|\leq\max\left\{|\Arg(\tau)|,\big|\Arg\log\left(1+\omega^{-1}\right)\big|\right\}<\theta,$$
since
\begin{align*}
\big|\Arg\log\left(1+\omega^{-1}\right)\big| 
	&=\bigg|\Arg(\omega^{-1})+\Arg\left(1-\omega^{-1}\left(\frac{1}{2}+\0\left(\omega^{-1}\right)\right)\right)\bigg| \\
	&<|\Arg(\omega)|<\theta.
\end{align*}
Note that the last inequality follows because $\Arg(\omega^{-1})$ and $\Arg(1-\omega^{-1})$ have opposite signs.  Therefore, 
$$\fc\left(\Otwh\right)\subset\Otwh\qquad\Rightarrow\qquad \fo\left(\Oxyh\right)\subset\Oxyh.$$
\end{proof}
Also the origin is in the boundary of $\Omega=l\left(\Oxyh\right)$, where $l$ is the bilinear map~\eqref{l}.  Given any $(x,y)\in\Oxyh$, we showed that $(x_n,y_n)\in\Oxyh, \forall n\in\N$.  Since $(x_n,y_n)=\psi_0^{-1}(u_n,v_n)$ and $|u_n|,|v_n|\to\infty$, the sequence $\{ (x_n,y_n)\}$ converges to the origin and so $(0,0)\in\partial\Oxyh\cap\partial\Omega$.   \\

To summarize, we have shown that there is a domain $\Omega$ with the origin in its boundary that is biholomorphic to $\Otwh\subset\C^2$ and on the latter, $f$ acts as the identity on the first coordinate and translation on the second.  What follows is a diagram of all the coordinate changes we performed to get that result:
\begin{diagram}
&\Omega &\rInto^f& \Omega &\  & & & & \\
&\dTo^{l^{-1}}_{\mbox{bilinear}}     &&\dTo^{l^{-1}}&& 						  &&  & \\
(x,y)\in&\Oxyh&\rInto^\fo&\Oxyh&\rInto^i &\psi_0^{-1}\left(\Ouv_2\right)&\rInto^i&\psi_0^{-1}\left(\Ouv\right) & \supset\Oxyh\ni(x_1,y_1)\\
\dMapsto^{\Psi}&\dTo^{\psi_0}_{\mbox{biholo.}}     &&\dTo^{\psi_0}&&\dTo^{\psi_0}							  &&\dTo^{\psi_0} & \ldMapsto[3,2] \\
&\Ouvh&\rInto^\fa&\Ouvh&\rInto^i&\Ouv_2 						    &\rInto^i&\Ouv& \\
&\dTo^{\psi_1}_{\mbox{biholo.}}    &&\dTo^{\psi_1}&&\dTo^{\psi_1}	  &&& \\
&\Ouwh&\rInto^\fb&\Ouwh&\rInto^i&\Ouw_2	& & &  \\
& \dTo^{\psi_2}_{\mbox{biholo.}}     &&\dTo^{\psi_2}&  		&				 &&&  \\
(\tau,\omega)\in &\Otwh&\rInto^\fc&\Otwh&\ni(\tau_1,\omega_1)=(\tau,\omega+1) &&&&
\end{diagram}

\section{Fatou-Bieberbach Domain}
Now that we have finished demonstrating Theorem A\ref{thmA}, we turn to Theorem B\ref{thmB}.  In order to apply our previous results, we want to assume that $f$ as in Theorem A is an automorphism.  From Theorem~\ref{WThm2.1.1}, due independently to Weickert \cite{W1} and Buzzard-Forstneric \cite{BF}, we know there exist automorphisms of $\C^2$ that approximate $f$ very closely near the origin.  In the power series expansions of $f$ and $\fo$ near the origin, we have only explicitly used terms up to degree at most $2k-1$.  By Theorem~\ref{WThm2.1.1}, there is an automorphism whose Taylor series expansion near the origin agrees with that of $f$ up through its degree $2k$ terms.  We now assume that $f$ is an automorphism of $\C^2$ since it is possible to have an automorphism $f$ as in Theorem A\ref{thmA}. Then $\fo=l^{-1}\circ f\circ l$ is an automorphism with the same general form as before in~\eqref{fo}.  Let 
$$\Sxyo:=\bigcup_{n\geq 0}\fo^{-n}\left(\Oxyh\right)\qquad\mbox{ and }\qquad
\s:=l\left(\Sxyo\right)=\bigcup_{n\geq 0}f^{-n}\left(l\left(\Oxyh\right)\right).$$

We can extend the domains of definition of $\omega$ and $\tau$ to $\Sxyo$:
\begin{align*}
\tau\circ\psi_1\circ\psi_0(x,y) &:=\tau\circ\psi_1\circ\psi_0(x_n,y_n), \\
\omega\circ\psi_0(x.y) &:=\omega\circ\psi_0(x_n,y_n)-n,
\end{align*}
where $\fo^n(x,y)=(x_n,y_n)\in\Oxyh$ for some $n\in\N$.  These are well-defined so we can use them to extend the domain of definition of $\Psi$ to $\Sxyo$: 
$$\Psi(x,y):=\psi_2\circ\psi_1\circ\psi_0(x,y)=\Psi(x_n,y_n)-(0,n),$$
where $\fo^n(x,y)=(x_n,y_n)\in\Oxyh$ for some $n\in\N$.  Since $\Psi$ is holomorphic on $\Oxyh$ and $\fo$ is biholomorphic on $\C^2$, this extension of $\Psi$ is holomorphic on $\Sxyo$.  In the following, we want to show that $\Psi\left(\Sxyo\right)=\C^2$.  
\begin{align*}
\pi_2\circ\Psi\left(\Sxyo\right)
	&=\omega\circ\psi_0\left(\Sxyo\right) \\
	&=\bigcup_{n\geq 0}\left\{\omega\circ\psi_0(x_n,y_n)-n \ | \ (x_n,y_n)\in\Oxyh\right\} \\
	&=\bigcup_{n\geq 0}\left\{\omega\left(\Ouvh\right)-n \right\} \\
	&=\bigcup_{n\geq 0}\left\{\pi_2\left(\Otwh\right)-n \right\} \\
	&=\C
\end{align*}
where the third equality follows since $\fo$ is an automorphism of $\C^2$ and the final equality follows from the definition of $\Otwh$ given in Proposition~\ref{psi2}.  For any $w\in\C$ and $n\in\N$, let 
$$W_n^w:=(\omega\circ\psi_0)^{-1}(w)\cap \fo^{-n}\left(\Oxyh\right).$$
Then $(\omega\circ\psi_0)^{-1}(w)\cap\Sxyo=\bigcup_{n\geq 0}W_n^w.$

\begin{thm}Fix any $w\in\pi_2\left(\Otwh\right)$.  Then
$$\tau\circ\psi_1\circ\psi_0=\pi_1\circ\Psi:(\omega\circ\psi_0)^{-1}(w)\cap\Sxyo\to\C$$
is a biholomorphism.\end{thm}
\begin{proof}
First we show that the extension of $\tau\circ\psi_1\circ\psi_0$ to $\Sxyo$ is holomorphic.  For any $(x,y)\in\Sxyo, \exists n\in\N$ such that $(x_n,y_n)\in\Oxyh$.  Let $U$ be a connected open neighborhood of $(x,y)$ small enough that $\fo^n(U)\subset\Oxyh$, which is then a connected open neighborhood of $(x_n,y_n)$.  Therefore  $\tau\circ\psi_1\circ\psi_0$ is holomorphic on $\fo^n(U)$.  Since $\fo$ is holomorphic and
$\tau\circ\psi_1\circ\psi_0=\tau\circ\psi_1\circ\psi_0\circ \fo^n$, it follows that $\tau\circ\psi_1\circ\psi_0$ is holomorphic on $U$.  Hence $\tau\circ\psi_1\circ\psi_0$ is holomorphic on $\Sxyo$. \\
Now we show injectivity.  Suppose $\tau\circ\psi_1\circ\psi_0(x,y)=\tau\circ\psi_1\circ\psi_0(\wt{x},\wt{y})$ for some $(x,y),(\wt{x},\wt{y})\in(\omega\circ\psi_0)^{-1}(w)\cap\Sxyo$.  For $n\in\N$ large enough we have $(x_n,y_n),(\wt{x}_n,\wt{y}_n)\in\Oxyh$ and
$$\tau\circ\psi_1\circ\psi_0(x_n,y_n)=\tau\circ\psi_1\circ\psi_0(x,y)
=\tau\circ\psi_1\circ\psi_0(\wt{x},\wt{y})=\tau\circ\psi_1\circ\psi_0(\wt{x}_n,\wt{y}_n).$$
Then
$$\pi_2\circ\Psi(x,y)=w=\pi_2\circ\Psi\left(\wt{x},\wt{y}\right)\qquad\Rightarrow\qquad
\pi_2\circ\Psi(x_n,y_n)=w+n=\pi_2\circ\Psi\left(\wt{x}_n,\wt{y}_n\right).$$
$\Psi$ is injective on $\Oxyh$ and $\fo$ is injective on $\C^2$, therefore
$$\Psi(x_n,y_n)=\Psi\left(\wt{x}_n,\wt{y}_n\right)\quad\Rightarrow\quad
(x_n,y_n)=\left(\wt{x}_n,\wt{y}_n\right)\quad\Rightarrow\quad
(x,y)=\left(\wt{x},\wt{y}\right).$$ \\
Finally we show surjectivity. \\
\textit{Claim:}  $\tau\circ\psi_1\circ\psi_0\left(W_n^w\right)
=\tau\circ\psi_1\circ\psi_0\left((\omega\circ\psi_0)^{-1}(w+n)\cap\Oxyh\right)$. \\
For any $(x,y)\in W_n^w$, it follows that $(x_n,y_n)\in\Oxyh$ and $\omega\circ\psi_0(x_n,y_n)=w+n$.  Hence $(x_n,y_n)\in(\omega\circ\psi_0)^{-1}(w+n)\cap\Oxyh$ and $\tau\circ\psi_1\circ\psi_0(x,y)=\tau\circ\psi_1\circ\psi_0(x_n,y_n).$ \\
Conversely, for any $(x,y)\in(\omega\circ\psi_0)^{-1}(w+n)\cap\Oxyh$, $\exists (\wt{x},\wt{y})\in\C^2$ such that $(\wt{x}_n,\wt{y}_n)=(x,y)$ since $\fo$ is an automorphism of $\C^2$.  
Then $\omega\circ\psi_0(\wt{x},\wt{y})=w$ so $(\wt{x},\wt{y})\in W_n^w$ and $\tau\circ\psi_1\circ\psi_0(x,y)=\tau\circ\psi_1\circ\psi_0(\wt{x},\wt{y}).$ \\
Therefore, $\tau\circ\psi_1\circ\psi_0(W_n^w)
=\tau\circ\psi_1\circ\psi_0\left((\omega\circ\psi_0)^{-1}(w+n)\cap\Oxyh\right).$ \\

Fix any $n\in\N$.  
\begin{align*}
\tau\circ\psi_1\circ\psi_0(W_n^w) 
	&=\tau\circ\psi_1\circ\psi_0\left((\omega\circ\psi_0)^{-1}(w+n)\cap\Oxyh\right) \\
	&=\tau\circ\psi_1\left((\pi_ 2\circ\psi_1)^{-1}(w+n)\cap\Ouvh\right) \\
 	&=\bigg\{\tau(u,w+n) \ \bigg| \ (u,w+n)\in\Ouwh\bigg\}  \\
	&=\left\{\tau-\log(w+n) \ \bigg| \ \re(\tau)>R,|\tau|<\left(\delta|w+n|\right)^\frac{k}{(k-1)(k+1)},
	|\Arg(\tau)|<\theta\right\},
\end{align*}
where the last equality follows because $\psi_2:\Ouwh\to\Otwh$ is a biholomorphism.  For fixed $w$, $|w+n|^\frac{k}{(k-1)(k+1)}$ grows much faster than $|\log(w+n)|$ as $n\to\infty$.  Therefore, 
$$\pi_1\circ\Psi\left((\omega\circ\psi_0)^{-1}(w)\cap\Sxyo\right)=\bigcup_{n\geq 0}\tau\circ\psi_1\circ\psi_0(W_n^w)=\C.$$
\end{proof}
For $n\in\N$, let
\begin{align*}
\Omega_n &:= \bigcup_{w\in\pi_2\left(\Otwh\right)} (\omega\circ\psi_0)^{-1}(w-n)\cap\Sxyo \\ 
\Psi_n &:=\left(\pi_1\circ\Psi\circ\fo^n,\pi_2\circ\Psi\right)=\Psi\circ\fo^n-(0,n).
\end{align*}
Then
$$\Psi_n:\Omega_n\to\C\times\left\{\pi_2\left(\Otwh\right)-n\right\}$$
is a biholomorphism and $\bigcup_{n\in\N}\Omega_n=\Sxyo$.  So $\left\{\Omega_n\right\}_{n\in\N}$ is an open cover of $\Sxyo$ with coordinate functions $\{\Psi_n\}_{n\in\N}$ which agree with each other on overlaps.  This defines on $\Sxyo$ a structure of a locally trivial fiber bundle with base $\C$ and fiber $\C$.  Hence,
$$\Psi:\Sxyo\to\C^2$$
is a biholomorphism and
$$\Psi(x,y)=\Psi\circ\fo^n(x,y)-(0,n)$$
for any $(x,y)\in\Sxyo$ and $n\in\N$.  In addition, $\fo$ acts as translation:
$$\Psi\circ\fo(x,y)=\Psi(x,y)+(0,1)$$
for any $(x,y)\in\Sxyo$.  Recall that $\s=l\left(\Sxyo\right)$ and let $\Phi:=\Psi\circ l^{-1}$.  In summary, we have the commutative diagram that illustrates Theorem B:
\begin{diagram}
\s & \rTo^f & \s \\
\dTo^\Phi &  & \dTo_\Phi \\ 
\C^2 & \rTo_{\Id+(0,1)} & \C^2
\end{diagram}
where all maps are biholomorphisms.\\


\begin{thebibliography}{BF}
\bibitem[A1]{A1}
M. Abate.
\newblock\emph{Holomorphic classification of 2-dimensional quadratic maps tangent to the identity}.
\newblock  Surikaisekikenkyusho Kokyuroku, \textbf{1447} (2005), 1-14.

\bibitem[A2]{A2}
M. Abate.
\newblock\emph{Discrete holomorphic local dynamical systems}.
\newblock  In "Holomorphic dynamical systems", Eds. G. Gentili, J. Guenot, G. Patrizio, Lect. Notes in Math. 1998, Springer, Berlin, 2010, 1-55.

\bibitem[AR]{AR}
M. Arizzi and J. Raissy.
\newblock\emph{On \'{E}calle-Hakim's theorems in holomorphic dynamics.}
\newblock \texttt{arXiv:1106.1710v2.} 

\bibitem[BF]{BF} G. Buzzard and F. Forstneric.
\newblock \emph{An interpolation theorem for holomorphic automorphisms of $\C^n$}.
\newblock J. of Geom. Analysis, \textbf{1} (2000), 101-108. 

\bibitem[CG]{CG} L. Carleson and T. Gamelin.
\newblock \emph{Complex dynamics}.
\newblock Springer-Verlag, New York, 1993.

\bibitem[H1]{H1} M. Hakim.
\newblock \emph{Transformations tangent to the identity.  Stable pieces of manifolds}.
\newblock Preprint, 1997.

\bibitem[H2]{H2} M. Hakim.
\newblock \emph{Analytic transformations of $(\C^p,0)$ tangent to the identity}.
\newblock Duke Mathematical Journal, \textbf{92} (1998), 403-428. 

\bibitem[M]{M} J. Milnor.
\newblock \emph{Dynamics in one complex variable}.
\newblock Annals of Mathematics Studies, Third Edition, 2006.

\bibitem[R]{R} M. Rivi.
\newblock \emph{Local behavior of discrete dynamical systems}.
\newblock Ph.D. Thesis, Universit\`{a} di Firenze, 1999.

\bibitem[SV]{SV}B. Stens\o nes and L. Vivas.
\newblock\emph{Basins of attraction of automorphism in $\C^3$}.
\newblock \texttt{arXiv:1111.2874.}

\bibitem[U]{U} T. Ueda.
\newblock Seminar Talk, Mt. Holyoke, 1994.

\bibitem[V1]{V1}L. Vivas.
\newblock\emph{Fatou-Bieberbach domains as basins of attraction of automorphisms tangent to the identity}.
\newblock \texttt{arXiv:0907.2061.}  

\bibitem[V2]{V2}L. Vivas.
\newblock\emph{Degenerate characteristic directions for maps tangent to the identity}.
\newblock \texttt{arXiv:1106.1471} (December 9, 2011).  

\bibitem[W1]{W1}B. Weickert.
\newblock\emph{Automorphisms of $\C^2$ tangent to the identity}.
\newblock Ph.D. Thesis, University of Michigan, 1997.

\bibitem[W2]{W2}B. Weickert.
\newblock\emph{Attracting basins for automorphisms of $\C^2$}.
\newblock Invent. Math. \textbf{132} (1998), 581-605. 

\end{thebibliography}
\end{document}